\documentclass[11pt]{article}
\usepackage{graphicx}
\usepackage{fullpage}
\usepackage{amsmath}
\usepackage{amssymb}
\usepackage{amsthm}
\usepackage{epsfig}
\usepackage{color}

\usepackage{tikz}
\usetikzlibrary{trees}
\tikzset{hide on/.code={\only<#1>{\color{fg!20}}}}

\definecolor{clemson-orange}{RGB}{234,106,32}
\definecolor{broncos-orange}{RGB}{252,76,2}
\definecolor{cincinnati-red}{RGB}{190,0,0}
\definecolor{pink}{RGB}{255,105,180}
\definecolor{celtics}{RGB}{46,123,59}
\definecolor{leafs-blue}{RGB}{0,58,120}
\definecolor{pure-cyan}{RGB}{0,100,92}
\definecolor{clemson-orange}{RGB}{234,106,32}
\definecolor{chicago-maroon}{RGB}{128,0,0}
\definecolor{northwestern-purple}{RGB}{82,0,99}
\definecolor{sauder-green}{RGB}{171,180,0}
\definecolor{saffron}{RGB}{255, 153, 51}
\definecolor{lawngreen}{RGB}{0,250,154}

\usepackage{mathtools}

\usepackage[square,numbers,sort&compress]{natbib}

\usepackage{enumitem}

\newcommand{\bb}{\mathbb}
\newcommand{\R}{\bb R}
\newcommand{\Z}{{\bb Z}}
\newcommand{\N}{{\bb N}}
\newcommand{\Q}{{\bb Q}}

\newcommand{\ch}{Chv\'{a}tal }

\DeclareMathOperator\proj{proj}

\DeclareMathOperator*{\conv}{conv}

\DeclareMathOperator*{\aff}{aff}

\DeclareMathOperator*{\linspan}{span}

\DeclareMathOperator*{\intt}{int}
\DeclareMathOperator*{\cone}{cone}
\DeclareMathOperator*{\ceil}{ceil}

\DeclareMathOperator*{\intcone}{intcone}

\DeclareMathOperator*{\carr}{carr}
\DeclareMathOperator*{\cc}{cc}

\theoremstyle{definition}
\newtheorem{theorem}{Theorem}[section]
\newtheorem{lemma}[theorem]{Lemma}

\newtheorem{prop}[theorem]{Proposition}
\newtheorem{claim}{Claim}
\newtheorem{question}{Question}
\newtheorem{definition}[theorem]{Definition}
\newtheorem{remark}[theorem]{Remark}

\newtheorem{example}[theorem]{Example}

\numberwithin{equation}{section}

\title{Mixed-integer linear representability, disjunctions, and \ch functions --- modeling implications\thanks{A preliminary version of this paper appeared as ``Mixed-integer linear representability, disjunctions, and variable elimination'' in the proceedings of the IPCO 2017 conference. See \cite{basu2017mixed} for a full reference.}}

\author{Amitabh Basu\footnote{Applied Mathematics and Statistics, Johns Hopkins University, USA. A. Basu was supported by the NSF grant CMMI1452820.}
\and Kipp Martin\footnote{Booth School of Business, University of Chicago}
\and Christopher Thomas Ryan\footnotemark[2]
\and Guanyi Wang\footnote{Industrial and Systems Engineering, Georgia Institute of Technology}}

\begin{document}

\maketitle

\begin{abstract}
Jeroslow and Lowe gave an exact geometric characterization of subsets of $\R^n$ that are projections of mixed-integer linear sets, also known as MILP-representable or MILP-R sets. We give an alternate algebraic characterization by showing that a set is MILP-R {\em if and only if} the set can be described as the intersection of finitely many {\em affine \ch inequalities} in continuous variables (termed AC sets). These inequalities are a modification of a concept introduced by Blair and Jeroslow.   Unlike the case for linear inequalities,  allowing for integer variables in \ch inequalities and projection does not enhance modeling power.  We show that the MILP-R sets are still precisely those sets that are modeled as affine \ch inequalites with integer variables. Furthermore, the projection of a set defined by affine \ch inequalites with integer variables is still an MILP-R set.
 We give a sequential variable elimination scheme that, when applied to a MILP-R set yields the AC set characterization. This is related to the elimination scheme of Williams and Williams-Hooker, who describe projections of integer sets using \emph{disjunctions} of affine \ch systems. We show that disjunctions are unnecessary  by showing how to find the affine \ch inequalities that cannot be discovered by the Williams-Hooker scheme. This allows us to answer a long-standing open question due to Ryan (1991) on designing an elimination scheme to represent finitely-generated integral monoids as a system of \ch inequalities \emph{without} disjunctions. Finally, our work can be seen as a generalization of the approach of Blair and Jeroslow, and Schrijver for constructing consistency testers for integer programs to general AC sets. 

\end{abstract}


\section{Introduction}\label{s:introduction}

Researchers are interested in characterizing sets that are projections of mixed-integer sets described by linear constraints. Such sets have been termed \emph{MILP-representable} or MILP-R  sets; see Vielma~\cite{vielma2015mixed} for a thorough survey. Knowing which sets are MILP-R is important because of the prevalence of good algorithms and software for solving MILP formulations. Therefore, if one encounters an application that can be modeled using MILP-R sets, then this sophisticated technology can be used to solve the application. There is also growing interest in generalizations of MILP-R sets, including projections mixed-integer points in a closed convex sets (see recent work by Del Pia and Poskin \cite{del2016mixed}, Dey, Diego and Mor\'an \cite{dey2013some}, and Lubin, Vielma and Zadik~\cite{lubin2017mixed,lubin2017regularity}). 

A seminal result of Jeroslow and Lowe \cite{jeroslow1984modelling} provides a geometric characterization of MILP-R sets as the sum of a finitely generated monoid, and a disjunction of finitely-many polytopes (see Theorem~\ref{theorem:jeroslow-lowe} below for a precise statement). 
Our point of departure is that we provide a constructive \emph{algebraic} characterization of MILP-representability that does not need disjunctions, but instead makes use of {\em affine \ch inequalities}, i.e. affine linear inequalities with rounding operations (for a precise definition see Definition~\ref{definition:ch-functions} below). We show that MILP-R sets are exactly those sets that satisfy a finite system of affine \ch inequalities, termed AC sets.

Affine \ch functions with continuous variables are a natural language for mixed-integer linear optimization.   Unlike the case for linear inequalities,  allowing for integer variables in \ch inequalities and projection does not enhance modeling power.  We show that the MILP-R sets are still precisely those sets are are modeled as affine \ch inequalites with integer variables. Furthermore, the projection of a set defined by affine \ch inequalites with integer variables is still an MILP-R set.


However, allowing disjunctions of sets   broadens the collection of sets that can be described.  There exist sets defined by disjunctions of affine \ch systems that are {\em not} MILP-R sets (see, for instance, Example~\ref{ex:dmic-is-too-big} below).
In other words, we show that disjunctions are not only unnecessary but are undesirable. This last message is underscored by the work of Williams \cite{williams-1,williams-2}, Williams and Hooker \cite{williams-hooker}, and Balas \cite{balas2011}. Their research attempts to generalize variable elimination methods for linear programming -- namely, the Fourier-Motzkin (FM) elimination procedure -- to integer programming problems. 
In these approaches, there is a need to introduce \emph{disjunctions} of inequalities that involve either rounding operations or congruence relations. Via this method, Williams, Hooker and Balas are able to describe the projections of integer sets as a \emph{disjunctive} system of affine \ch inequalities. The introduction of disjunctions is a point in common between the existing elimination methods of Williams, Hooker and Balas and the geometric understanding of projection by Jeroslow-Lowe. However, disjunctions in general can be unwieldy. Moreover, as stated above, allowing disjunctions together with affine \ch inequalities (as done in the algebraic approaches of Williams, Hooker and Balas) takes us out of the realm of MILP-R sets. 

Our approach to characterizing MILP-R sets is related to  \emph{consistency testers} for pure integer programs. Given a rational matrix $A$, a consistency tester is a function that takes as input a vector $b$ and returns a value that indicates whether the set $\left\{x : Ax \ge b, x \text{ integer} \right\}$ is non-empty. Seminal work by Blair and Jeroslow \cite{blair82} constructs a consistency tester for integer programs that is a pointwise maximum of a set of finitely many \ch functions (termed a {\em Gomory function} in Blair and Jeroslow \cite{blair82}). In \cite{schrijver86}, Schrijver obtains a version of this result that builds on the concepts of the \ch rank and total dual integrality of an integer system. A consistency tester describes a special type of MILP-R set; the projection of the pairs $(x,b)$ where $Ax \ge b$ onto the space of $b$'s. Our work generalizes the approach of Schrijver \cite{schrijver86} to apply to not only mixed-integer linear sets, but more generally to AC sets (and by our main result, MILP-R sets). 

Finally,  in Section~\ref{s:variable-elimination}  we give a  ``lift-and-project'' variable elimination scheme for mixed-integer AC sets. Our scheme, as opposed to the ones proposed by Williams and Hooker, and Balas, does not need to resort to disjunctions. Towards this end, our new procedure introduces auxiliary integer variables to simplify the structure of the AC system. In this transformed system, the projection of integer variables is easier to do without introducing disjunctions; at this stage, we project out the auxiliary variables that were introduced, as well as the variables that were originally intended to be eliminated. When this method is applied to a mixed-integer {\em linear} set, it generates redundant linear inequalities which, when combined with ceiling operators, characterize the projection without the  need for disjunctions. This is our proposed extension of Fourier-Motzkin elimination to handle integer variables, without using disjunctions.

{\em In summary, we are able to simultaneously show four things: 1) disjunctions are not necessary for mixed-integer linear representability (if one allows affine \ch inequalities), an operation that shows up in both the Jeroslow-Lowe and the Williams-Hooker approaches, 2) the language of affine \ch functions is a robust one for integer programming, being closed under integrality and projection, 3) our algebraic characterization comes with a variable elimination scheme unlike the geometric approach of Jeroslow-Lowe, and 4) our algebraic characterization is exact, as opposed to the algebraic approach of Williams-Hooker which does not yield a complete characterization of MILP-R sets.
}

Moreover, our algebraic characterization could be useful to obtain other insights into the structure of MILP-R sets that is not apparent from the geometric perspective. As an illustration, we resolve an open question posed in Ryan \cite{ryan1991} on the representability of integer monoids using our characterization. Theorem 1 in \cite{ryan1991} shows that every finitely-generated integer monoid can be described as a finite system of \ch inequalities but leaves open the question of how to construct the associated \ch functions via elimination. Ryan states that the elimination methods of Williams in \cite{williams-1,williams-2} do not address her question because of the introduction of disjunctions. Our work provides a constructive approach for finding a \ch inequality representation of finitely-generated integer monoids using elimination (see Section~\ref{s:variable-elimination}). 

Our new algebraic characterization may also lead to novel algorithmic ideas where researchers optimize by directly working with affine \ch functions, rather than using traditional branch-and-cut or cutting plane methods. We also believe the language of affine \ch functions has potential for modeling applied problems, since the operation of rounding affine inequalities has an inherent logic that may be understandable for particular applications. We leave both of these avenues as directions for future research.

The paper is organized as follows. Section~\ref{s:preliminaries} introduces our key definitions -- including mixed-integer linear representability and affine \ch functions -- used throughout the paper. It also contains a statement and intepretation of our main result, making concrete the insights described in this introduction.  This includes the definitions of MILP-representability and  affine \ch functions. Section~\ref{s:MILP-as-MIC} contains the proof of our main result.
Section~\ref{s:consistency-testers} relates  our work to the existing literature of consistency testers for integer programs, which was  the  source of inspiration for  this paper. Finally, Section~\ref{s:variable-elimination} explores our methodology from the perspective of elimination of integer variables, where we compare and contrast approach with the existing methodologies of Williams, Hooker and Balas. 
Section~\ref{s:conclusion}  has concluding remarks.


\section{Definitions and discussion of main result}\label{s:preliminaries}

In this section we introduce the definitions and notation needed to state our main result. We also provide an intuitive discussion of the implications of the result. 

$\N, \Z, \Q, \R$ denote the set of natural numbers, integers, rational numbers and reals, respectively. Any of these sets subscripted by a plus means the nonnegative elements of that set. For instance, $\Q_+$ is the set of nonnegative rational numbers. The ceiling operator $\lceil a \rceil$ gives the smallest integer no less than $a \in \R$. The projection operator $\proj_Z$ where $Z \subseteq \left\{x_1, \dots, x_n\right\}$ projects a vector $x \in \R^n$ onto the coordinates in $Z$. We use the notation $x_{-i}$ to denote the set $\left\{x_1, \dots, x_{i-1}, x_{i+1}, \dots, x_n \right\}$ and thus $\proj_{x_{-i}}$ refers to projecting out the $i$-th variable.  The following \emph{classes of sets}  are used throughout the paper.

An \emph{LP set} (short for linear programming set) is any set defined by the intersection of finitely many linear inequalities.%
%
%
\footnote{Of course, an LP-set is nothing other than a polyhedron. We use the terminology LP-set for the purpose of consistency with the definitions that follow.}
%
%
More concretely, $S \subseteq \R^n$ is an LP set if there exists an $m \in \N$, matrix $A \in \Q^{m \times n}$, and vector $b \in \R^m$ such that $S = \left\{ x \in \R^n : Ax \ge b \right\}$. We denote the collection of all LP sets by (LP). 

A set that results from applying the projection operator to an LP-set is called a \emph{LP-R set} (short for linear programming representable set). The set  $S \subseteq \R^n$ is an LP-R set if there exists an $m, p \in \N$, matrices $B \in \Q^{m \times n}$, $C \in \Q^{m \times d}$ and vector $b \in \R^m$ such that $S  = \proj_x \{ (x, y) \in \R^n  \times \R^p :  B x  + C y \ge b   \} .$ We denote the collection of all LP-R sets by (LP-R).

It well known that any LP-R set is an LP set, i.e., the projection of a polyhedron is also a polyhedron (see, for instance, Chapter~2 of \cite{martin1999large}). The typical proof uses Fourier-Motzkin (FM) elimination, a technique that is used in this paper as well (see the proof of Theorem~\ref{theorem:b-j-mod} and Section~\ref{s:consistency-testers} below). FM elimination is a method to eliminate (and consequently project out) continuous variables from a system of linear inequalities. For a detailed description of the FM elimination procedure we refer the reader to Martin \cite{martin1999large}. We provide some basic notation for the procedure here. FM elimination takes as input a linear system $Ax \ge b$ where $x = (x_1, \dots, x_n)$ and produces row vectors $u_1^1, \dots, u_1^{t_1}$ (called \emph{Fourier-Motzkin multipliers}) such that $u_1^j A x_{-1} \ge u_1^j b$ for $j = 1, \dots, t_1$ describes $\proj_{x_{-1}} \left\{x : Ax \ge b\right\}$.  This procedure can be applied iteratively to sequentially eliminate variables. When all variables are eliminated we denote the corresponding FM multipliers by $u^1, \dots, u^t$. We  make reference to FM multipliers at various points in the paper.

We  introduce both collections (LP) and (LP-R) (even though they are equal) to emphasize the point that, in general,  projecting sets could lead to a larger family, as in some of the other classes of sets defined below. 

A \emph{MILP set} (short for mixed-integer linear programming set) is any set defined by the intersection of finitely many linear inequalities where some or all of the variables in the linear functions defining the inequalities are integer-valued. The set  $S \subseteq \R^n$ is a MILP set if there exists an $m \in \N$, $I \subseteq \{1,2,\dots, n\}$, matrix $A \in \Q^{m \times n}$, and vector $b \in \R^m$ such that $S = \left\{ x \in \R^n : Ax \ge b, x_j \in \Z \text{ for } j \in I \right\}$. The collection of all MILP sets is denoted (MILP). 

Following Jeroslow and Lowe \cite{jeroslow1984modelling}, we define an \emph{MILP-R} set (short for mixed-integer linear programming representable set) to be any set that results from applying a projection operator to an MILP set.  The set $S \subseteq \R^n$ is an MILP-R set if there exists an $m, p, q \in \N$, matrices $B \in \Q^{m \times n}$, $C \in \Q^{m \times p}$ and $D \in \Q^{m \times q}$ and vector $b \in \R^m$ such that
\begin{equation*}
S  =  \proj_x \left\{ (x,y,z) \in \R^n \times \R^p \times \Z^q : Bx + Cy + Dz \ge b \right\}.
\end{equation*} 
The collection of all MILP-R sets is denoted (MILP-R). 

It is also well-known that there are MILP-R sets that are not MILP sets (see Williams~\cite{williams-1} for an example). Thus, projection  provides more modeling power when using integer variables, as opposed to the LP and LP-R sets where variables are all real-valued.

The key result known in the literature about MILP-R sets uses the following concepts. Given a finite set of vectors $\{r^1, \dots, r^t\}$, $\cone \{r^1, \dots, r^t\}$ is the set of all nonnegative linear combinations, and $\intcone \{r^1, \dots, r^t\}$ denotes the set of all nonnegative {\em integer} linear combinations. The set $\intcone \{r^1, \dots, r^t\}$ is also called a \emph{finitely-generated integer monoid} with generators $\{r^1, \dots, r^t\}$. The following is the main result from Jeroslow and Lowe~\cite{jeroslow1984modelling} stated as Theorem 4.47 in Conforti et. al. \cite{conforti2014integer}.

\begin{theorem}\label{theorem:jeroslow-lowe}
A set $S \subset \R^n$ is MILP-representable if and only if there exists rational polytopes $P_1, \dots, P_k \subseteq \R^n$ and vectors $r^1, \dots, r^t \in \Z^n$ such that
\begin{equation}\label{eq:jeroslow-lowe-chraracterization}
S = \bigcup_{i=1}^k P_i + \intcone \left\{r^1, \dots, r^t\right\}.
\end{equation}
\end{theorem}

This result is a \emph{geometric} characterization of  MILP-R sets.
We provide an alternative {\em algebraic} characterization of MILP-R sets using  {\em affine \ch functions and inequalities}. {\em \ch functions}, first introduced by Blair and Jeroslow \cite{blair82}, are obtained by taking nonnegative combinations of linear functions and using the ceiling operator. We extend this original definition to allow for \emph{affine} functions to define {\em affine \ch functions}. To make this distinction precise we formally define affine \ch functions using the concept of finite binary trees from Ryan~\cite{ryan1986thesis}.

\begin{definition}\label{definition:ch-functions}
An affine \ch function $f:\R^n\to \R$ is constructed as follows. We are given a finite binary tree where each node of the tree is either: (i) a leaf, which corresponds to an affine linear function on $\R^n$ with rational coefficients; (ii) has one child with corresponding edge labeled by either a $\lceil \cdot \rceil $ or a number in $\Q_+$, or (iii) has two children, each with edges labeled by a number in $\Q_+$.

Start at the root node and (recursively) form functions corresponding to subtrees rooted at its children. If the root has a single child whose subtree is $g$, then either (a) $f = \lceil g \rceil $ if the corresponding edge is labeled $\lceil \cdot \rceil $ or (b) $f = \alpha g$ if the corresponding edge is labeled by $a \in \Q_+$.
If the root has two children with corresponding edges labeled by $a \in \Q_+$ and $b \in \Q_+$ then $f = ag + bh$ where $g$  and $h$  are  functions corresponding to the respective children of the root.\footnote{The original definition of \ch function in Blair and Jeroslow \cite{blair82} does not employ binary trees. Ryan shows the two definitions are equivalent in~\cite{ryan1986thesis}.}

The \emph{depth} of a binary tree representation $T$ of an affine \ch function is the length of the longest path from the root to a node in $T$, and ceiling count $\cc(T)$ is the total number of edges labeled $\lceil \cdot \rceil$. \quad $\triangleleft$
\end{definition}


\begin{example}
Below, $\hat{f}$ is a \ch function and $\hat{g}$ is an affine \ch function:
\begin{align}
  \hat{f} & = 3 \lceil x_1 + 5 \lceil 2 x_1 + x_2 \rceil \rceil + \lceil 2 x_3 \rceil, \\
  \hat{g} & = 3 \lceil x_1 + 5 \lceil 2 x_1 - x_2 + 3.5 \rceil \rceil + \lceil -2 x_3 \rceil.
\end{align}
See Figure \ref{example:binary-tree} for a binary tree representation $T(\hat{g})$ of the affine \ch function $\hat{g}$.
\begin{figure}
\begin{center}
\begin{tikzpicture}[level distance = 4.5em, sibling distance=15em,
  every node/.style = {shape=rectangle, rounded corners,
    draw, align=center,
    top color=white, bottom color=white},
    edge from parent/.style = {draw, -latex}]
  \node {$\hat{g}(x) = 3 \lceil x_1 + 5 \lceil 2 x_1 - x_2 + 3.5 \rceil \rceil+ \lceil -2 x_3 \rceil$}
    child { node {$\lceil x_1 + 5 \lceil 2 x_1 - x_2 + 3.5 \rceil \rceil$}
      child{ node{$x_1 + 5 \lceil 2 x_1 - x_2 + 3.5 \rceil$}
      child{ node{$x_1$}
      edge from parent node {1}
      }
      child{ node{$\lceil 2 x_1 - x_2 + 3.5 \rceil$}
      child{ node{$2 x_1 - x_2 + 3.5$}
      edge from parent node {$\lceil \cdot \rceil$}
      }
      edge from parent node {5}
      }
      edge from parent node {$\lceil \cdot \rceil$}
      }
      edge from parent node {3}
      }
    child { node {$\lceil -2 x_3 \rceil$}
      child{node{ $-2x_3$ }
      edge from parent node {$\lceil \cdot \rceil$}
      }
      edge from parent node {1}
       };
\end{tikzpicture}
\end{center}
\caption{Binary tree structure for affine \ch function}
\label{example:binary-tree}
\end{figure}
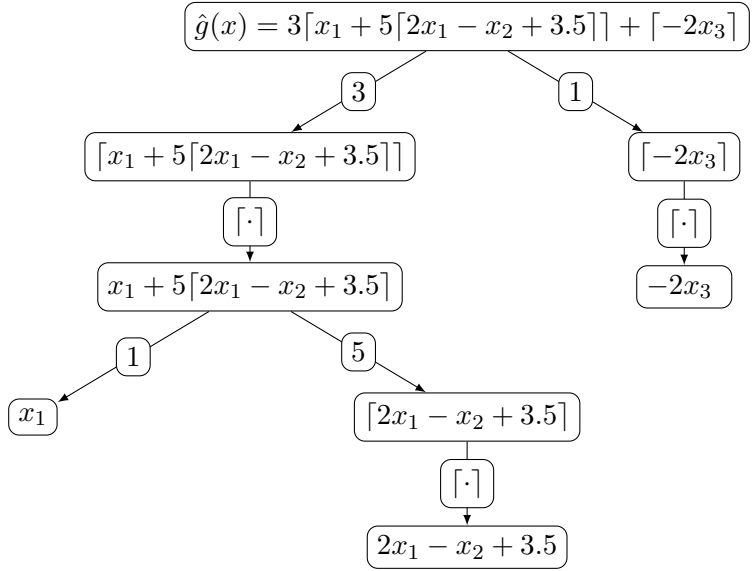
This representation has depth $4$ and ceiling count $\cc(T(\hat{g})) = 3$.
\end{example}

The original definition of {\em \ch function} in the literature requires the leaves of the binary tree to be linear functions, and the domain of the function to be $\Q^n$ (see \cite{ryan1986thesis,blair82,ryan1991}). Our definition above allows for {\em affine} linear functions at the leaves, and the domain of the functions to be $\R^n$. We use the term {\em \ch function} to refer to the setting where the leaves are linear functions. In this paper, the domain of all functions is $\R^n$. This change to the domain from $\Q^n$ to $\R^n$ is not just cosmetic; it is imperative for deriving our results. See also the discussion after Theorem~\ref{theorem:b-j-mod}.

\begin{definition}\label{definition:AC-inequality}
An inequality $f(x) \leq b$, where $f$ is an affine \ch function and $b\in \R$, is called an {\em affine \ch inequality}. 
\end{definition}

\begin{remark} Note that if $f$ is an affine \ch function, it does not necessarily mean that $-f$ is also an affine \ch function. Because of this, the inequality $f(x) \geq b$ is, in general, not an affine \ch inequality: the direction of the inequality matters in Definition~\ref{definition:AC-inequality}. \quad $\triangleleft$
\end{remark}

An \emph{AC set} (short for affine \ch set) is any set defined by the intersection of finitely affine \ch inequalities.  The set  $S \subseteq \R^n$ is an AC set if there exists an $m \in \N$, affine \ch function $f_1, f_2, \dots, f_m$, and a real vector $b \in \R^m$ such that $ S  =  \{ x \in \R^n  : f_i(x) \le b_i, i = 1, 2, \ldots, m  \}$. The collection of all AC sets is denoted (AC). 

A set that results from applying the projection operator to an AC set is called an \emph{AC-R set} (short for affine \ch representable set).  The set  $S \subseteq \R^n$ is an AC-R set if there exists an $m, p \in \N$, affine \ch functions $f_1, f_2, \dots, f_m$ defined on $\R^{n + p}$, and a vector $b \in \R^m$ such that $S  = \proj_x \{ (x, y) \in \R^n  \times \R^p  :  f_i(x, y) \le b_i, i = 1, 2, \ldots, m  \}$. The collection of all AC-R sets is denoted (AC-R). 

An \emph{MIAC set} (short for mixed-integer affine \ch set) is an affine \ch set where some of the variables involved are integer.   The set  $S \subseteq \R^n$ is an MIAC set if there exists an $m \in \N$, $I \subseteq \{1,2,\dots, n\}$, affine \ch functions $f_1, f_2, \dots, f_m$, and vector $b \in \R^m$ such that $S  = \{ x \in \R^n  : f_i(x) \le b_i, i = 1, 2, \ldots, m, \,  x_{j}  \in \Z, \, j \in I \}$. The collection of all MIAC sets is denoted (MIAC). 

A set that results from applying the projection operator to an MIAC set is called a \emph{MIAC-R} set (short for mixed-integer affine \ch representable set. The set $S \subseteq \R^n$ is an MIAC-R set if there exists an $m, p, q \in \N$, affine \ch functions $f_1, f_2, \dots, f_m$ defined on $\R^{n+p+q}$ and vector $b \in \R^m$ such that
\begin{equation*}
S  =  \proj_x \left\{ (x,y,z) \in \R^n \times \R^p \times \Z^q : f_i(x, y, z) \le b_i, i = 1, 2. \dots, m \right\}.
\end{equation*} 
The collection of all MIAC-R sets is denoted (MIAC-R).

Finally, a \emph{DMIAC set} (short for disjunctive mixed-integer affine \ch set) is a set that can be written as the disjunction of finitely many MIAC sets. The set  $S\subseteq \R^n$ is a DMIAC set if there exists a $m,t \in \N$, affine functions $f_{ki}$ for $k =1, 2, \dots, t$ and $i = 1, 2, \dots, n$, subsets $I_k \subseteq \{1,2,\dots,n\}$ for $k = 1, 2, \dots, t$ such that 
\begin{equation*}
S  =  \bigcup_{k = 1}^t \{ x \in \R^n  : f_{ki}(x) \le b_i, i = 1, 2, \ldots, m, \,  x_{kj}  \in \Z,  \, j \in I_{k} \}
\end{equation*}
The collection of all DMIAC sets is denoted (DMIAC). 

We now  state the main theorem in this paper.

\begin{theorem}\label{theorem:main} The following relationships exist among the sets defined above
\begin{equation*}
(\text{LP}) =   (\text{LP-R})  \subsetneq  (\text{MILP})  \subsetneq  (\text{MILP-R})  =   (\text{AC})   =  (\text{AC-R})  = (\text{MIAC})  =  (\text{MIAC-R}) \subsetneq (\text{DMIAC}).
\end{equation*}
\end{theorem}

The first three relationships in Theorem~\ref{theorem:main} are well-known in the optimization community. The key insights in our paper are the remaining relationships, i.e.,   
\begin{equation*}
(\text{MILP-R})  =  (\text{AC}) = (\text{AC-R}) = (\text{MIAC}) =  (\text{MIAC-R}) \subsetneq (\text{DMIAC}).
\end{equation*}
Section~\ref{ss:all-put-together} contains our proof of Theorem~\ref{theorem:main}, which builds on the results of Sections~\ref{ss:guanyi-trick-direction} and~\ref{ss:theory-heavy-direction}. Before going to the proof, we examine Theorem~\ref{theorem:main} and draw out its implications. This makes precise some of the informal discussion we had in the introduction. 

\begin{enumerate}[label=(\roman*)]
\item The relationship $(\text{MILP-R}) = (\text{AC})$ provides our algebraic characterization of mixed-integer linear representability. Note, in particular, that the class (AC) does not allow disjunctions.
\item The relationship $(\text{AC}) = (\text{MIAC-R})$ shows that adding integer variables and projecting an AC set does not yield additional modeling power.
\item The relationship $(\text{MILP-R}) \subsetneq (\text{DMIAC})$ shows that combining disjunctions with affine \ch inequalites does describe a strictly larger collection of sets than can be described (even through projection) by linear equalities with integer variables.
\end{enumerate}

Point (i) provides  a ``disjunction-free'' characterization of mixed-integer representability. Together, points (ii) and (iii) suggest that (AC) is a natural algebraic language for mixed-integer linear programming. The collection (AC) uses continuous variables, with no disjunctions, to describe all MILP sets and their projections, whereas (DMIAC) takes us outside of the realm of mixed-integer programming.

\section{The modeling power of \ch inequalities}\label{s:MILP-as-MIC}

This section contains the proof of our main result Theorem~\ref{theorem:main}. The proof is broken up across three subsections. The first two subsections provide careful treatment is the two most challenging containments to establish: $\text{(MIAC)} \subseteq \text{(MILP-R)}$ (the content of Section~\ref{ss:guanyi-trick-direction}) and $\text{(MILP-R)} \subseteq \text{(AC)}$ (the content of Section~\ref{ss:theory-heavy-direction}). Finally, in Section~\ref{ss:all-put-together}, the pieces are put together in a formal proof of Theorem~\ref{theorem:main}.

\subsection{MIAC sets are MILP-R sets }\label{ss:guanyi-trick-direction}

We show how to ``lift'' an MIAC set to a mixed-integer linear set. The idea is simple -- replace ceiling operators with additional integer variables. However, we need to work with an appropriate representation of an affine \ch function in order to implement this idea. The next result provides the correct representation.

\begin{theorem} \label{theorem:PossibleChvatal}
For every affine \ch function $f$ represented by a binary tree $T$, one of the following cases hold:
\begin{description}
 \item[Case 1:] $\cc(T) = 0$, which implies that $f$ is an affine linear function.
 \item[Case 2:] $f = \gamma \lceil g_1 \rceil + g_2$, where $\gamma >0$ and $g_1, g_2$ are affine \ch functions such that there exist binary tree representations $T_1, T_2$ for $g_1, g_2$ respectively,  with $\cc(T_1) + \cc(T_2) + 1 \leq \cc(T)$.
 \end{description}
\end{theorem}

\begin{proof}
We use induction on the depth of the binary tree $T$. For the base case, if $T$ has depth $0$, then $\text{cc}(T) = 0$ and we are in Case 1. The inductive hypothesis assumes that for some $k \geq 0$, every affine \ch function $f$ with a binary tree representation  $T$ of depth less or equal to $k$, can be expressed in Case 1 or 2.

For the inductive step, consider an affine \ch function $f$ with a binary tree representation $T$ of depth $k + 1$. If the root node of $T$ has a single child, let $T'$ be the subtree of $T$ with root node equal to the child of the root node of $T$. We now consider two cases: the edge at the root node is labeled with a $\lceil\cdot\rceil$, or the edge is labeled with a scalar $\alpha > 0$. In the first case, $f = \lceil g \rceil$ where $g$ is an affine \ch function which has $T'$ as a binary tree representation. Also, $\cc(T') + 1 = \cc(T)$. Thus, we are done by setting $g_1 = g$, $g_2 = 0$ and $\gamma = 1$. In the second case, $f = \alpha g$ where $g$ is an affine \ch function which has $T'$ as a binary tree representation, with $\cc(T') = \cc(T)$. Note that $T'$ has smaller depth than $T$. Thus, we can apply the induction hypothesis on $g$ with representation $T'$. If this ends up in Case 1, then $0 = \cc(T') = \cc(T)$ and $f$ is in Case 1. Otherwise, we obtain $\gamma' > 0$, affine \ch functions $g'_1, g'_2$, and binary trees $T'_1, T'_2$ representing $g'_1, g'_2$ respectively, with \begin{equation}\label{eq:cc-single-edge}\cc(T'_1) + \cc(T'_2) + 1 \leq \cc(T') = \cc(T)\end{equation} such that $g = \gamma'\lceil g'_1 \rceil + g'_2$. Now set $\gamma = \alpha \gamma'$, $g_1 = g'_1$, $g_2 = \alpha g'_2$, $T_1 = T'_1$ and $T_2$ to be the tree whose root node has a single child with $T'_2$ as the subtree, and the edge at the root labeled with $\alpha$. Note that $\cc(T_2) = \cc(T'_2)$. Also, observe that $T_1, T_2$ represents $g_1, g_2$ respectively. Combined with~\eqref{eq:cc-single-edge}, we obtain that $\cc(T_1) + \cc(T_2) + 1 \leq \cc(T)$.

If the root node of $T$ has two children, let $S_1, S_2$ be the subtrees of $T$ with root nodes equal to the left and right child, respectively, of the root node of $T$. Then, $f = \alpha h_1 + \beta h_2$, where $\alpha, \beta>0$ and $h_1, h_2$ are affine \ch functions with binary tree representations $S_1, S_2$ respectively. Also note that the depths of $S_1, S_2$ are both strictly less than the depth of $T$, and

\begin{equation}\label{eq:cc-h1-h2-f}\cc(S_1) + \cc(S_2) = \cc(T)\end{equation}

By the induction hypothesis applied to $h_1$ and $h_2$ with representations $S_1, S_2$, we can assume both of them end up in Case 1 or 2 of the statement of the theorem. If both of them are in Case 1, then $\cc(S_1) = \cc(S_2) = 0$, and by~\eqref{eq:cc-h1-h2-f}, $\cc(T) = 0$. So $f$ is in Case 1.

Thus, we may assume that $h_1$ or $h_2$ (or both) end up in Case 2. There are three subcases, (i) $h_1, h_2$ are both in Case 2, (ii) $h_1$ is Case 2 and $h_2$ in Case 1, or (iii) $h_2$ in Case 2 and $h_1$ in Case 1. We analyze subcase (i), the other two subcases are analogous. This implies that there exists $\gamma'>0$, and affine \ch functions $g'_1$ and $g'_2$ such that $h_1 = \gamma' \lceil g'_1 \rceil + g'_2$, and there exist binary tree representations $T'_1, T'_2$ for $g'_1, g'_2$ respectively, such that \begin{equation}\label{eq:cc-g'1-g'2-h1}\cc(T'_1) + \cc(T'_2) + 1 \leq \cc(S_1).\end{equation} Now set $\gamma = \alpha\gamma'$, $g_1(x) = g'_1(x)$ and $g_2(x) = \alpha g'_2(x) + \beta h_2(x)$. Then $f = \gamma\lceil g_1 \rceil + g_2$. Observe that $g_2$ has a binary tree representation $T_2$ such that the root node of $T_2$ has two children: the subtrees corresponding to these children are $T'_2$ and $S_2$, and the edges at the root node of $T_2$ are labeled by $\alpha$ and $\beta$ respectively. Therefore, \begin{equation}\label{eq:cc-g2-g'2-h2}\cc(T_2) \leq \cc(T'_2) + \cc(S_2).\end{equation}
Moreover, we can take $T_1 = T'_1$ as the binary tree representation of $g_1$. We observe that \begin{equation*}\begin{array}{rcl} \cc(T_1) + \cc(T_2) + 1 &\leq &\cc(T'_1) + \cc(T'_2) + \cc(S_2) + 1 \\ & \leq & \cc(S_1) + \cc(S_2) =  \cc(T)\end{array}\end{equation*} where the first inequality is from the fact that $T_1 = T'_1$ and~\eqref{eq:cc-g2-g'2-h2}, the second inequality is from~\eqref{eq:cc-g'1-g'2-h1} and the final equation is~\eqref{eq:cc-h1-h2-f}.\end{proof}

For an MIAC set, where each associated affine \ch function is represented by a binary tree, the \emph{total ceiling count of this representation} is the sum of the ceiling counts of all these
binary trees. The next lemma shows how to reduce the total ceiling count of a MIAC set by one, in exchange for an additional integer variable.

\begin{lemma}\label{lem:reduce-ceil-count}
Given a system $C=\{(x,z) \in \R^n\times\mathbb{Z}^q: ~ f_i(x,z) \leq b_i\}$ of affine \ch inequalities with a total ceiling count $c \geq 1$, there exists a system $P=\{(x,z,\bar{z}) \in \R^n\times\mathbb{Z}^q \times \Z: ~ f'_i(x,z) \leq b'_i\}$ of affine \ch inequalities with a total ceiling count of at most $c-1$, and $C = \proj_{(x,z)}(P)$.
\end{lemma}

\begin{proof} Since $c\geq 1$, at least one of the $f_i$ has a binary tree representation $T$ with a strictly positive ceiling count. Without loss of generality we assume it is $f_1$. This means $f_1$, along with its binary tree representation $T$, falls in Case 2 of Theorem~\ref{theorem:PossibleChvatal}. 
Therefore, one can write $f$ as $f_1 = \gamma \lceil g_1 \rceil + g_2,$ with $\gamma >0$, and $g_1, g_2$ are affine \ch functions such that there exist binary tree representations $T_1, T_2$ for $g_1, g_2$ respectively,  with $\cc(T_1) + \cc(T_2) + 1 \leq \cc(T)$. Dividing by $\gamma$ on both sides, the inequality $f_1(x,z) \leq b_1$ is equivalent to $
        \lceil g_1(x,z) \rceil +(1/\gamma) g_2(x,z) \leq b_1/\gamma.$ Moving $(1/\gamma) g_2(x,z)$ to the right hand side, we get $
        \lceil g_1(x,z) \rceil \leq -(1/\gamma) g_2(x,z) + b_1/\gamma.$
      This inequality is easily seen to be equivalent to two inequalities, involving an extra integer variable $\bar{z} \in \mathbb{Z}$: $
        \lceil g_1(x,z) \rceil \leq \bar{z} \leq - (1/\gamma) g_2(x,z) + b_1 / \gamma, \label{eq:with-ceiling-case1}$
      which, in turn is equivalent to $ g_1(x,z) \leq \bar{z} \leq -(1/\gamma) g_2(x,z) + b_1 / \gamma,$ since $\bar{z} \in \mathbb{Z}$. Therefore, we can replace the constraint $f_1(x,z) \leq b_1$ with the two constraints
      \begin{align}
        & g_1(x,z) - \bar{z} \leq 0, \label{eq:g1}\\
        & (1/\gamma) g_2(x,z) + \bar{z} \leq b_1/ \gamma  \Leftrightarrow g_2(x,z) + \gamma \bar{z} \leq b_1\label{eq:g2}
      \end{align}
as long as we restrict $\bar{z} \in \Z$. Note that the affine \ch functions on the left hand sides of~\eqref{eq:g1} and~\eqref{eq:g2} have binary tree representations with ceiling count equal to $\cc(T_1)$ and $\cc(T_2)$ respectively. Since $\cc(T_1) + \cc(T_2) + 1 \leq \cc(T)$, the total ceiling count of the new system is at least one less than the total ceiling count of the previous system.
\end{proof}

The key result of this subsection is an immediate consequence.

\begin{theorem}\label{theorem:mic-is-milp}
Every MIAC set is a MILP-R set. That is, $(\text{MIAC}) \subseteq (\text{MILP-R})$. 
\end{theorem}

\begin{proof} Consider a system $S=\{(x,z) \in \R^n\times\mathbb{Z}^q: ~ f_i(x,z) \leq b_i\}$ of affine \ch inequalities describing the MIAC set, with total ceiling count $c \in \N$. Apply Lemma~\ref{lem:reduce-ceil-count} at most $c$ times to get a system $S' = \{(x,z,z') \in \R^n \times \Z^q\times \Z^m: Ax + Bz + Cz' \geq d\}$ such that $S = \proj_{(x,z)}(S')$, where $m \leq c$. The problem is that the $z$ variables are integer constrained in the system describing $S'$, and the definition of MILP-representability requires the target space -- $(x,z)$ in this case -- to have no integer constrained variables. This can be handled in a simple way. Define $S'' := \{(x,z,z',v)  \in \R^n \times \R^q\times \Z^m \times \Z^q) : Ax + Bz + Cz' \geq d, \;\; z = v \}$ with additional integer variables $v$, and observe that $S' = \proj_{(x,z,z')}(S'')$ and thus, $S = \proj_{(x,z)}(S'')$. Since $x,z$ are continuous variables in the system describing $S''$, we obtain that $S$ is MILP-representable.
\end{proof}


\begin{example}. We give an example, showing the above procedure at work. Consider the AC set
$$C = \{(x_1, x_2, x_3, x_4) \in \Z: f(x) = \lceil 3x_1 + 2.5x_2\rceil + \lceil\lceil 0.5x_3 \rceil  - 0.8x_4\rceil \leq 0\}.$$

Add variable $y_1 \in \Z$ and the constraints

$$\lceil\lceil 0.5x_3 \rceil  - 0.8x_4\rceil \leq y_1 \leq -\lceil 3x_1 + 2.5x_2\rceil.$$

Remove the outer $\lceil \cdot \rceil$ on the left hand side to obtain

$$\lceil 0.5x_3 \rceil  - 0.8x_4 \leq y_1 \leq -\lceil 3x_1 + 2.5x_2\rceil,$$ which gives two affine \ch inequalities:

\begin{equation}\label{eq:step-1} 
\begin{array}{l}
\lceil 0.5x_3 \rceil  - 0.8x_4 - y_1 \leq 0 \\
y_1 +\lceil 3x_1 + 2.5x_2\rceil \leq 0
\end{array}
\end{equation}

Taking the first affine \ch inequality in~\eqref{eq:step-1}, and introducing another variable $y_2 \in \Z$, we obtain 

$$\lceil 0.5x_3 \rceil  \leq y_2 \leq + 0.8x_4 + y_1$$
and removing the $\lceil\cdot\rceil$ on the left hand side, we obtain

$$0.5x_3 \leq y_2 \leq + 0.8x_4 + y_1,$$ giving rise to two new affine \ch functions:

\begin{equation}\label{eq:step-2-1} 
\begin{array}{l}
0.5x_3- y_2 \leq 0 \\
y_2 - 0.8x_4 - y_1 \leq 0
\end{array}
\end{equation}

Similarly, processing the second affine \ch inequality in~\eqref{eq:step-1}, we obtain two new affine \ch inequalities involving a new variable $y_3\in \Z$:

\begin{equation}\label{eq:step-2-1} 
\begin{array}{l}
3x_1 + 2.5x_2- y_3 \leq 0 \\
y_3 + y_1 \leq 0
\end{array}
\end{equation}

So we finally obtain that 
\begin{equation*}
C =\proj_{(x_1, \ldots, x_4)} \left\{(x_1,x_2,x_3,x_4,y_1,y_2,y_3,y_4): \begin{array}{lll} 0.5x_3- y_2  &\leq &0 \\ - 0.8x_4 - y_1 +y_2 & \leq & 0 \\
3x_1 + 2.5x_2- y_3 &\leq &0 \\
y_1 + y_3& \leq &0
\end{array}\right\}. \quad \triangleleft
\end{equation*}
\end{example}

\subsection{MILP-R sets are MIAC sets}\label{ss:theory-heavy-direction}

This direction leverages some established theory in integer programming, in particular,

\begin{theorem}[cf. Corollary 23.4 in Schrijver~\cite{schrijver86}]\label{theorem:b-j-mod}
For any rational $m \times n$ matrix $A$, there exists a finite set of Chv\'atal functions $f_i: \R^m \to \R$, $i\in I$ with the following property: for every $b\in \R^m$, $\{ z \in \Z^n  \, : \, Az \ge b\}$ is nonempty if and only if $f_i(b) \leq 0$ for all $i\in I$. Moreover, these functions can be explicitly constructed from the matrix $A$.
\end{theorem}

The main difference between Corollary 23.4 in~\cite{schrijver86} and Theorem~\ref{theorem:b-j-mod} is that we allow the right hand side $b$ to be nonrational.%
%
%
\footnote{We say a vector is \emph{nonrational} if it has at least one component that is not a rational number. We use this terminology instead of \emph{irrational}, which we take to mean having no rational components.}
%
%
This difference is indispensable in our analysis (see the proof of Theorem~\ref{theorem:milp-is-mic}). Although our proof of Theorem~\ref{theorem:b-j-mod} is conceptually similar to the approach in Schrijver~\cite{schrijver86}, we need to handle some additional technicalities related to irrationality. In particular, we extend the supporting results used to prove Corollary 23.4b(i) in Schrijver~\cite{schrijver86} to the nonrational case.
To our knowledge, no previous work has explicitly treated the case where $b$ is nonrational.

Theorem~\ref{theorem:b-j-mod} in the rational case was originally obtained by Blair and Jeroslow in \cite[Theorem~5.1]{blair82}), but used a different methodology. This work in turn builds on seminal work on integer programming duality by Wolsey in~\cite{wolsey1981,wolsey1981integer}. Wolsey showed that the family of subadditive functions suffices to give a result like Theorem~\ref{theorem:b-j-mod}; Blair and Jeroslow improved this to show that the smaller family of \ch functions suffice.

To prove Theorem~\ref{theorem:b-j-mod} we need some preliminary definitions and results. A system of linear inequalities $Ax \ge b$ where $A = (a_{ij})$ has $a_{ij} \in \Q$ for all $i,j$ (that is, $A$ is rational) is \emph{totally dual integral} (TDI) if the maximum in the LP-duality equation
\begin{equation*}
\min \{c^\top x : Ax \ge b \} = \max \{y^\top b : A^\top y = c, y \ge 0\}
\end{equation*}
has an integral optimal solution $y$ for every integer vector $c$ for which the minimum is finite. Note that rationality of $b$ is not assumed in this definition. When $A$ is rational, the system $Ax \ge b$ can be straightforwardly manipulated so that all coefficients of $x$ on the right-hand side are integer. Thus, we may often assume without loss that $A$ is integral.

For our purposes, the significance of a system being TDI is explained by the following result. For any polyhedron $P \subseteq \R^n$, $P'$  denotes its {\em \ch closure}\footnote{The \ch closure of $P$ is defined in the following way. For any polyhedron $Q\subseteq \R^n$, let $Q_I := \conv(Q \cap \Z^n)$ denote its integer hull. Then $P' := \bigcap\{H_I: H \textrm{ is a halfpsace containing }P\}$.}. We also recursively define the $t$-th \ch closure of $P$ as $P^{(0)} := P$, and $P^{(t+1)} = (P^{(t)})'$ for $i \geq 1$.

\begin{theorem}[See Schrijver \cite{schrijver86}  Theorem 23.1]\label{theorem:schrijver-23.1}
Let $P = \{  x \, : \, Ax \ge b \}$ be nonempty and assume $A$ is integral.  If $Ax \ge b$ is a TDI representation of the polyhedron $P$ then
\begin{eqnarray}
P' = \{  x \, : \, Ax \ge  \lceil b \rceil\}.  \label{eq:tdi-ch-closure}
\end{eqnarray}
\end{theorem}

We now show how to manipulate the system $Ax \ge b$ to result in one that is TDI. The main power comes from the fact that this manipulation depends only on $A$ and works for every right-hand side $b$.

\begin{theorem}\label{theorem:basu}
Let $A$ be a rational $m \times n$ matrix.  Then there exists another nonnegative $q \times m$ rational matrix $U$ such that  for every $b \in \R^m$ the polyhedron $P = \{  x \in \R^{n} \, : \,  Ax \ge b \},$  has a representation $P = \{  x \in \R^{n} \, : \,  Mx \ge b' \}$ where the system $Mx \ge b'$ is TDI and $M = UA,$  $b' = Ub.$
\end{theorem}

\begin{proof}
First construct the matrix $U.$  Let  ${\cal P}(\{1, 2, \ldots, m \})$  denote the power set of  $\{1, 2, \ldots, m\}.$   For each subset of rows $a^i$  of  $A$ with $i  \in S$  where    $S \in {\cal P}(\{1, 2, \ldots, m \}),$   define the cone
\begin{eqnarray}
C(S) := \{  a \in \R^n \, : \,  a = \sum_{i \in S} u_i a^i,  u _i \ge 0,  i \in S \}. \label{eq:defineCS}
\end{eqnarray}
By construction the cone $C(S)$ in~\eqref{eq:defineCS}  is a rational polyhedral cone.  Then by Theorem 16.4 in Schrijver \cite{schrijver86} there exist integer vectors  $m^{k},$ for $k = k^{S}_{1},  k^{S}_{2} \ldots, k^{S}_{q_{S}}$   that define a  Hilbert basis for this cone.   In this indexing scheme  $q_{S}$ is the cardinality of the set $S.$  We assume  that there are  $q_{S}$ distinct indexes $ k^{S}_{1},  k^{S}_{2} \ldots, k^{S}_{q_{S}}$  assigned to each set $S$ in the power set  ${\cal P}(\{1, 2, \ldots, m \}).$  Since  each $m^{k} \in C(S)$  there is  a nonnegative nonnegative  vector $u^{k}$ that generates $m^{k}$. Without loss  each $u^{k}$ is an $m-$dimensional vector since we can assume a component of zero for each component $u^{k}$ not indexed by  $S.$   Thus   $u^{k} A = m^{k}.$   Define a matrix  $U$ to be the matrix with rows $u^{k}$ for all $k =   k^{S}_{1},  k^{S}_{2} \ldots, k^{S}_{q_{S}}$ and $S \in {\cal P}(\{1, 2, \ldots, m \}).$   Then $M = UA$ is a matrix with rows corresponding to all of the Hilbert bases for the power set of $\{ 1, 2, \ldots, m \}.$  That is, the number of rows in $M$ is $q = \sum_{S \in  {\cal P}(\{1, 2, \ldots, m \})} q_{S}.$
\vskip 5pt
We first show that $M x \ge b'$ is a TDI representation of
\begin{eqnarray}
P =  \{  x \in \R^{n} \, : \,  Ax \ge b \}  = \{  x \in \R^{n} \, : \,  Mx \ge b' \}.  \label{eq:equal-poly}
\end{eqnarray}
Note that  $\{  x \in \R^{n} \, : \,  Ax \ge b \}$  and $\{  x \in \R^{n} \, : \,  Mx \ge b' \}$  define the same polyhedron since the system of the inequalities $Mx \ge b'$ contains all of the inequalities $Ax \ge b$ (this is because the power set of $\{1, 2, \ldots, m \}$ includes each singleton set) plus additional inequalities that are nonnegative aggregations of inequalities in the system  $Ax \ge b.$  In order to show $Mx \ge b'$ is a TDI representation, assume $c \in \R^n$ is an integral vector and the minimum of
\begin{eqnarray}\label{eq:tdi}
\max \{ y b' \, : \,  y M = c, y \ge 0 \}
\end{eqnarray}
is finite. It remains to show there is an integral optimal dual solution to \eqref{eq:tdi}. By linear programming duality $\min \{ c x | M x \ge b' \}$ has an optimal solution $\bar x$ and
\begin{eqnarray}
\max \{ y b' \, : \,  y M = c, y \ge 0 \} =  \min \{ c x \, : \,  M x \ge b' \}.  \label{eq:tdi-1}
\end{eqnarray}
Then by equation~\eqref{eq:equal-poly}
\begin{eqnarray}
 \min \{ c x \, : \,  M x \ge b' \} =  \min \{ c x \, : \,  Ax \ge b\}.   \label{eq:tdi-2}
\end{eqnarray}
and    $\min \{ c x \, : \,  A x \ge b \}$  also has optimal solution  $\bar x.$ Then again by linear programming duality
\begin{eqnarray}
 \min \{ c x \, : \,  Ax \ge b\}  = \max \{ u b  \,  :  \,  u A = c,  \, u \ge 0\}.
\end{eqnarray}
Let $\bar u$ be an optimal dual solution to $ \max \{ u b  \,  :  \,  u A = c,  \, u \ge 0\}.$  Let $i$  index the strictly positive $\bar u_i$  and define  $S = \{ i \, : \,  \bar u_i > 0 \}.$   By construction of  $M$  there is a subset of rows of $M$ that form a Hilbert basis for $C(S).$  By construction of $C(S),$ $\bar u A = c$ implies $c \in C(S).$  Also, since $\bar u$ is an optimal dual solution,  it must satisfy complementary slackness. That is, $\bar u_{i} > 0$ implies that  $ a^{i} \bar x = b_i.$  Therefore $S$ indexes a set of tight constraints in the system $A \bar x \ge b.$ Consider an arbitrary element  $m^{k}$ of the Hilbert basis associated with the cone $C(S)$. There is a corresponding
  $m-$vector $u^{k}$ with support in $S$ and
\begin{equation*}
u^{k} A \bar x = \sum_{i \in S} u^{k}_{i}a^i \bar x = \sum_{i \in S} u^{k}_{i} b_i = u^{k} b = b'_{k}.
\end{equation*}
Since $u^{k} A = m^{k}$ and $u^{k} b = b'_{k}$ we have
\begin{equation}\label{eq:this-is-useful}
 m^{k} \bar x = b'_{k},   \quad \forall k = k^{S}_{1},  k^{S}_{2} \ldots, k^{S}_{q_{S}}.
\end{equation}
As argued above, $c \in C(S)$ and is, therefore, generated by nonnegative integer multiples of the $m^{k}$ for  $k = k^{S}_{1},  k^{S}_{2} \ldots, k^{S}_{q_{S}}.$   That is, there exist nonnegative integers $\bar y_{k}$  such that
\begin{equation}\label{eq:write-out-m-k}
c = \sum_{k  =  k^{S}_{1}}^{k^{S}_{q_{S}}} \bar y_{k} m^{k}.
\end{equation}

Hence there exists a nonnegative $q$-component integer vector $\bar y$ with support contained in the  set  indexed by  $k^{S}_{1},  k^{S}_{2} \ldots, k^{S}_{q_{S}}$
 such that
\begin{equation}\label{eq:write-out-c}
c = \bar y M.
\end{equation}
Since $\bar y  \ge 0,$  $\bar y$ is feasible to the left hand side of~\eqref{eq:tdi-1}.   We use \eqref{eq:this-is-useful} and \eqref{eq:write-out-m-k} to show
\begin{equation}\label{eq:zero-duality-gap}
\bar y b' = c \bar x,
\end{equation}
which implies that $\bar y$ is an optimal integral dual solution to \eqref{eq:tdi} (since $\bar x$  and  $\bar y$ are  primal-dual feasible), implying the result.

To show \eqref{eq:zero-duality-gap},  use the fact that the support of $\bar y$ is  contained in the  set  indexed by  $k^{S}_{1},  k^{S}_{2} \ldots, k^{S}_{q_{S}}$ which implies
\begin{eqnarray}
\bar y b' = \sum_{k  =  k^{S}_{1}}^{k^{S}_{q_{S}}} \bar y_{k} b'_{k}.
\end{eqnarray}
Then by \eqref{eq:this-is-useful}  substituting $m^{k} \bar x$ for $b'_{k}$ gives
\begin{eqnarray}
\bar y b' = \sum_{k  =  k^{S}_{1}}^{k^{S}_{q_{S}}} \bar y_{k} b'_{k} =  \sum_{k  =  k^{S}_{1}}^{k^{S}_{q_{S}}} \bar y_{k} m^{k} \bar x.
\end{eqnarray}
Then by \eqref{eq:write-out-m-k} substituting $c$ for  $\sum_{k  =  k^{S}_{1}}^{k^{S}_{q_{S}}} \bar y_{k} m^{k} $  gives
\begin{eqnarray}
\bar y b' = \sum_{k  =  k^{S}_{1}}^{k^{S}_{q_{S}}} \bar y_{k} b'_{k} =  \sum_{k  =  k^{S}_{1}}^{k^{S}_{q_{S}}} \bar y_{k} m^{k} \bar x = c \bar x.
\end{eqnarray}

This gives  \eqref{eq:zero-duality-gap} and completes the proof.\end{proof}

\begin{remark}
When $S$ is a singleton set, i.e. $S = \{ i\}$, the corresponding $m^{k}$ for $k = k^{S}_{1}$ may be a scaling of the corresponding $a^{i}.$  This does not affect our argument that~\eqref{eq:equal-poly} holds. \quad $\triangleleft$
\end{remark}

\begin{remark}
Each of the $m^{k}$ vectors that define each Hilbert basis may be assumed to be integer.  Therefore if $A$ is an integer matrix,  $M$ is an integer matrix. \quad $\triangleleft$
\end{remark}

Next we will also need a series of results about the interaction of lattices and convex sets.

\begin{definition} Let $V$ be a vector space over $\R$. A {\em lattice} in $V$ is the set of all integer combinations of a linearly independent set of vectors $\{a^1, \ldots,  a^m\}$ in $V.$ The set $\{a^1, \ldots,  a^m\}$ is called the basis of the lattice. The lattice is {\em full-dimensional} if it has a basis that spans $V$. \quad $\triangleleft$
\end{definition}

\begin{definition} Given a full-dimensional lattice $\Lambda$ in a vector space $V$, a {\em $\Lambda$-hyperplane} is an affine hyperplane $H$ in $V$ such that $H = \aff(H \cap \Lambda)$. This implies that in $V = \R^n$, if $H$ is a $\Z^n$-hyperplane, then $H$ must contain $n$ affinely independent vectors in $\Z^n$. \quad $\triangleleft$
\end{definition}

\begin{definition} Let $V$ be a vector space over $\R$
and let $\Lambda$ be a full-dimensional lattice in $V$. Let $\mathcal{H}_{\Lambda}$ denote the set of all $\Lambda$-hyperplanes that contain the origin.
Let $C \subseteq V$ be a convex set. Given any $H \in \mathcal{H}_{\Lambda}$, we say that the {\em $\Lambda$-width of $C$ parallel to $H$}, denoted by $\ell(C,\Lambda, V, H)$, is the total number of distinct $\Lambda$-hyperplanes parallel to $H$ that have a nonempty intersection with $C$. The {\em lattice-width} of $C$ with respect to $\Lambda$ is defined as $\ell(C,\Lambda, V) := \min_{H \in \mathcal{H}_{\Lambda}} \ell(C,\Lambda, V, H)$. \quad $\triangleleft$
\end{definition}

We will need this classical ``flatness theorem" from the geometry of numbers -- see Theorem VII.8.3 on page 317 of Barvinok~\cite{barvinok}, for example. 

\begin{theorem}\label{thm:flatness} Let $V\subseteq \R^n$ be a vector subspace with $\dim(V) = d$, and let $\Lambda$ be a full-dimensional lattice in $V$. Let $C \subseteq V$ be a compact, convex set. If $C \cap \Lambda = \emptyset$, then $\ell(C,\Lambda, V) \leq d^{5/2}$.
\end{theorem}
%
%

We will also need a theorem about the structure of convex sets that contain no lattice points in their interior, originally stated in Lovasz~\cite{Lovasz89}.

\begin{definition} A convex set $S\subseteq \R^n$ is said to be {\em lattice-free} if $\intt(S)\cap \Z^n = \emptyset$. A {\em maximal lattice-free set} is a lattice-free set that is not properly contained in another lattice-free set. \quad $\triangleleft$
\end{definition}

\begin{theorem}[Theorem~1.2 in Basu et. al. \cite{bccz} and also Lovasz~\cite{Lovasz89}]\label{thm:mlfc-structure}
A set $S\subset \R^n$ is a maximal lattice-free convex set  if and only if one of the following holds:
\begin{itemize}
\item[(i)] $S$ is a polyhedron of the form $S= P+L$ where $P$ is a polytope, $L$ is a rational linear space, $\dim(S)=\dim(P)+\dim(L)=n$, $S$ does not contain any integral point in its interior and there is an integral point in the relative interior of each facet of $S$;
\item[(ii)] $S$ is an irrational affine hyperplane of $\R^n$.
\end{itemize}

\end{theorem}

The previous result is used to prove the following.

\begin{theorem}\label{thm:flatness-unbdd-poly} Let $A\in \R^{m \times n}$ be a rational matrix. Then for any $b \in \R^m$ such that $P:= \{x \in \R^n : Ax \geq b\}$ satisfies $P \cap \Z^n = \emptyset$, we must have $\ell(P,\Z^n, \R^n) \leq n^{5/2}$. Note that $P$ is not assumed to be bounded.
\end{theorem}

\begin{proof} If $P$ is not full-dimensional, then $\aff P$ is given by a system $\{ x : \tilde A x = \tilde b \}$ where the matrix $\tilde A$ is rational, since the matrix $A$ is rational and $\tilde A$ can be taken to be a submatrix of $A$. Now take a $\Z^n$-hyperplane $H$ that contains $\{ x | \tilde A x = 0 \}$. Then $\ell(P, \Z^n, \R^n, H) = 0$ or $1$, depending on whether the translate in which $P$ is contained in a $\Z^n$-hyperplane translate of $H$ or not. This immediately implies that $\ell(P,\Z^n,\R^n)$ is either $0, 1$.

Thus, we focus on the case when $P$ is full-dimensional. By Theorem~\ref{thm:mlfc-structure}, there exists a basis $v^1, \ldots, v^n$ of $\Z^n$, a natural number $k \leq n$, and a polytope $C$ contained in the linear span of $v^1, \ldots, v^k$, such that $P \subseteq C + L$, where $L = \linspan(\{v^{k+1}, \ldots, v^n\})$ and $(C + L) \cap \Z^n = \emptyset$ (the possibility of $k = n$ is allowed, in which case $L = \{0\}$).

Define $V = \linspan(\{v^1, \ldots, v^k\})$ and $\Lambda$ as the lattice formed by the basis $\{v^1, \ldots, v^k\}$. Since $C$ is a compact, convex set in $V$ and $C\cap \Lambda = \emptyset$, by Theorem~\ref{thm:flatness}, we must have that $\ell(C, \Lambda, V) \leq k^{5/2}$. Every $\Lambda$-hyperplane $H \subseteq V$ can be extended to a $\Z^n$-hyperplane $H' = H + L$ in $\R^n$. This shows that $\ell(C + L, \Z^n, \R^n) \leq k^{5/2} \leq n^{5/2}$. Since $P \subseteq C + L$, this gives the desired relation that $\ell(P,\Z^n, \R^n) \leq n^{5/2}$. \end{proof}

\begin{example}\label{ex:infinite-lattice-width}
If $A$ is not rational, the above result is not true. Consider the set
\begin{equation*}
P := \{(x_1, x_2) : x_2 = \sqrt{2}(x_1-1/2)\}
\end{equation*}
Now, $P \cap Z^2 = \emptyset$. Any $\Z^2$-hyperplane containing $(0,0)$ is the span of some integer vector. All such hyperplanes intersect $P$ in exactly one point, since the hyperplane defining $P$ has an irrational slope and so intersects every $\Z^2$-hyperplane in exactly one point. Hence, $\ell(P, \Z^2, \R^2) = \infty$ for all $H \in \mathcal H_{\Z^2}$ and so $\ell(P, \Z^2, \R^2) = \infty$. \quad $\triangleleft$
\end{example}

Theorem~\ref{thm:flatness-unbdd-poly} will help to establish bounds on the Chv\'atal rank of any lattice-free polyhedron with a rational constraint matrix. First we make the following modification of equation (6) on page 341 in Schrijver~\cite{schrijver86}.

\begin{lemma}\label{lem:face-chvatal} Let $A\in \R^{m \times n}$ be a rational matrix. Let $b \in \R^m$ (not necessarily rational) and let $P:= \{x \in \R^n : Ax \geq b\}$. Let $F \subseteq P$ be a face. Then $F^{(t)} = P^{(t)} \cap F$ for any $t \in \N$.
\end{lemma}

\begin{proof} The proof follows the proof of (6) in Schrijver ~\cite{schrijver86} on pages 340-341 very closely. As observed in Schrijver~\cite{schrijver86}, it suffices to show that $F' = P' \cap F$.

Without loss of generality, we may assume that the system $Ax \geq b$ is TDI (if not, then throw in valid inequalities to make the description TDI). Let $F = P \cap \{x : \alpha x = \beta\}$ for some integral $\alpha \in \R^n$.
The system $Ax \geq b, \alpha x \geq \beta$ is also TDI, which by Theorem 22.2 in Schrijver~\cite{schrijver86} implies that the system $Ax \leq b, \alpha x = \beta$ is also TDI (one verifies that the proof of Theorem 22.2 does not need rationality for the right hand side).

Now if $\beta$ is an integer, then we proceed as in the proof of the Lemma at the bottom of page 340 in Schrijver~\cite{schrijver86}.

If $\beta$ is not an integer, then $\alpha x \geq \lceil \beta \rceil$ and $\alpha x \leq \lfloor \beta \rfloor$ are both valid for $F'$, showing that $F' = \emptyset$. By the same token, $\alpha x \geq \lceil \beta \rceil$ is valid for $P'$. But then $P' \cap F = \emptyset$ because $\lceil \beta \rceil > \beta$. \end{proof}

We now prove the following modification of Theorem 23.3 from Schrijver~\cite{schrijver86}.

\begin{theorem}\label{thm:schrijver-23.3} For every $n\in \N$, there exists a number $t(n)$ such that for any rational matrix $A\in \R^{m \times n}$ and $b \in \R^m$ (not necessarily rational) such that $P:= \{x \in \R^n : Ax \geq b\}$ satisfies $P \cap \Z^n = \emptyset$, we must have $P^{(t(n))} = \emptyset$.
\end{theorem}

\begin{proof} We closely follow the proof in Schrijver~\cite{schrijver86}. The proof is by induction on $n$. The base case of $n = 1$ is simple with $t(1) = 1$. Define $t(n) := n^{5/2} + 2 + (n^{5/2} + 1)t(n-1)$.

Since $P \cap \Z^n = \emptyset$, $\ell(P,\Z^n,\R^n) \leq n^{5/2}$ by Theorem~\ref{thm:flatness-unbdd-poly}, this means that there is an integral vector $c\in \R^n$ such that
 \begin{equation}\label{eq:width-bound}
 \lfloor \max_{x\in P} c^Tx\rfloor - \lceil\min_{x\in P} c^Tx \rceil \leq n^{5/2}.
 \end{equation}

Let $\delta = \lceil\min_{x\in P} c^Tx \rceil$. We claim that for each $k = 0, \ldots, n^{5/2} + 1$, we must have \begin{equation}\label{eq:slice} P^{(k+1 +k\cdot t(n-1))} \subseteq \{x : c^Tx \geq \delta + k\}.\end{equation}

For $k=0$, this follows from definition of $P'$. Suppose we know \eqref{eq:slice} holds for some $\bar k$; we want to establish it for $\bar k + 1$. So we assume $P^{(\bar k+1 +\bar k\cdot t(n-1))} \subseteq \{x : c^Tx \geq \delta + \bar k\}$. Now, since $P\cap \Z^n = \emptyset$, we also have $P^{(\bar k+1 +\bar k\cdot t(n-1))} \cap \Z^n = \emptyset$. Thus, the face $F = P^{(\bar k+1 +\bar k\cdot t(n-1))} \cap \{x : c^T x = \delta + \bar k\}$ satisfies the induction hypothesis and has dimension strictly less than $n$. By applying the induction hypothesis on $F$, we obtain that $F^{t(n-1)} = \emptyset$. By Lemma~\ref{lem:face-chvatal}, we obtain that $P^{(\bar k+1 +\bar k\cdot t(n-1) + t(n-1))} \cap \{x : c^T x = \delta + \bar k \} = \emptyset$. Thus, applying the \ch closure one more time, we would obtain that $P^{(\bar k+ 1 +\bar k\cdot t(n-1) + t(n-1) + 1)} \subseteq \{x : c^Tx \geq \delta + \bar k + 1) \}$. This confirms~\eqref{eq:slice} for $\bar k + 1$.

Using $k = n^{5/2} + 1$ in~\eqref{eq:slice}, we obtain that $P^{(n^{5/2}+2 +(n^{5/2} + 1)\cdot t(n-1))} \subseteq \{x : c^Tx \geq \delta + n^{5/2} + 1\}$. From~\eqref{eq:width-bound}, we know that $\max_{x\in P} c^Tx < \delta + n^{5/2} + 1$. This shows that $P^{(n^{5/2}+2 +(n^{5/2} + 1)\cdot t(n-1))} \subseteq P \subseteq \{x : c^Tx < \delta + n^{5/2} + 1\}$. This implies that $P^{(n^{5/2}+2 +(n^{5/2} + 1)\cdot t(n-1))} = \emptyset$, as desired.\end{proof}

This allows us to establish the following.

\begin{theorem}[c.f. Theorem~23.4 in Schrijver~\cite{schrijver86}]\label{theorem:schrijver-23.4}  For each rational matrix $A$ there exists a positive integer $t$ such that for every right hand side vector $b$ (not necessarily rational),
\begin{eqnarray}
\{ x \, : \, Ax \ge b \}^{(t)} = \{ x \, : \, Ax \ge b \}_{I}. \label{eq:t-bound}
\end{eqnarray}
\end{theorem}

\begin{proof} The proof proceeds exactly as the proof of Theorem 23.4 in Schrijver~\cite{schrijver86}. The proof in Schrijver~\cite{schrijver86} makes references to Theorems 17.2, 17.4 and 23.3 from Schrijver~\cite{schrijver86}. Every instance of a reference to Theorem 23.3 should be replaced with a reference to Theorem~\ref{thm:schrijver-23.3} above. Theorems 17.2 and 17.4 do not need the rationality of the right hand side.
\end{proof}

We now have all the machinery we need to prove Theorem~\ref{theorem:b-j-mod}.

\begin{proof}[Proof of Theorem~\ref{theorem:b-j-mod}]
Given $A$ we can generate a nonnegative matrix $U$ using Theorem~\ref{theorem:basu} so that $UAz \ge Ub$ is TDI for all $b.$  Then by Theorem~\ref{theorem:schrijver-23.1} we get the \ch closure using the system  $UAz \ge \lceil Ub \rceil.$ Using Theorem~\ref{theorem:schrijver-23.4} we can apply this process $t$ times independent of $b$ and know we end up with $\{ z \, : \, Az \ge b \}_{I}.$  We then apply Fourier-Motzkin elimination to this linear system and the desired $f_i$'s are obtained.
\end{proof}

With Theorem~\ref{theorem:b-j-mod} in hand we can now prove the main theorem of this subsection. This uses the following straightforward lemma that is stated without proof.

\begin{lemma}\label{lem:affine-comp-chvatal}
Let $T:\R^{n_1} \to \R^{n_2}$ be an affine transformation involving rational coefficients, and let $f:\R^{n_2} \to \R$ be an affine \ch function. Then $f\circ T :\R^{n_1} \to \R$ can be expressed as $f\circ T(x) = g(x)$ for some affine \ch function $g : \R^{n_1}\to \R$.
\end{lemma}


\begin{theorem}\label{theorem:milp-is-mic}
Every MILP-R set is an AC set. That is, $(\text{MILP-R}) \subseteq (\text{AC})$.  
\end{theorem}

\begin{proof}
Let $m, n,  p, q \in \N$. Let $A \in \Q^{m \times n}, B \in \Q^{m\times p}, C \in \Q^{m\times q}$ be any rational matrices, and let $d\in \Q^m$. Define $
\mathcal{F} = \{(x,y,z) \in \R^n \times \R^p \times \Z^q \;: \;\;Ax + By + Cz \geq d\}.$
It suffices to show that the projection of $\mathcal{F}$ onto the $x$ space is an AC set.

By applying Fourier-Motzkin elimination on the $y$ variables, we obtain rational matrices $A', C'$ with $m'$ rows for some natural number $m'$, and a vector $d'\in \Q^{m'}$ such that the projection of $\mathcal{F}$ onto the $(x,z)$ space is given by $\overline{\mathcal{F}}:= \{(x,z) \in \R^n \times \Z^q \;: \;\;A'x + C'z \geq d'\}.$

Let $f_i:\R^{m'}\to \R$, $i\in I$ be the set of Chv\'atal functions obtained by applying Theorem~\ref{theorem:b-j-mod} to the matrix $C'$. 
It suffices to show that the projection of $\overline{\mathcal{F}}$ onto the $x$ space is $\hat{\mathcal{F}}:=\{x\in \R^n \;: \; f_i(d' - A'x) \leq 0, \;\; i\in I\}$
since for every $i\in I$, $ f_i(d' - A'x) \leq 0$ can be written as $g_i(x) \leq 0$ for some affine Chv\'atal function $g_i$, by Lemma~\ref{lem:affine-comp-chvatal}.\footnote{This is precisely where we need to allow the arguments of the $f_i$'s to be nonrational because the vector $d' - A'x$ that arise from all possible $x$ is sometimes nonrational.} This follows from the following sequence of equivalences.
\begin{equation*}
\begin{array}{rcl}
x \in \proj_x(\mathcal{F}) & \Leftrightarrow & x \in \proj_x(\overline{\mathcal{F}}) \\
& \Leftrightarrow & \exists z \in \Z^{q} \textrm{ such that } (x,z) \in \overline{\mathcal{F}} \\
& \Leftrightarrow & \exists z \in \Z^{q}  \textrm{ such that } C'z \geq d' -A'x \\
& \Leftrightarrow & f_i(d' -A'x) \leq 0 \textrm{ for all } i\in I \qquad (\textrm{By Theorem~\ref{theorem:b-j-mod}}) \\
& \Leftrightarrow & x \in \hat{\mathcal{F}}. \qquad (\textrm{By definition of } \hat{\mathcal{F}})
\qquad \qquad \qquad \qquad \qquad \qquad \qquad \qquad \qquad \ \ \qedhere
\end{array}
\end{equation*}
\end{proof}

\begin{remark}\label{rem:homogeneous} We note in the proof of Theorem~\ref{theorem:milp-is-mic} that if the right hand side $d$ of the mixed-integer set is $0$, then the affine \ch functions $g_i$ are actually \ch functions. This follows from the fact that the function $g$ in Lemma~\ref{lem:affine-comp-chvatal} is a \ch function if $f$ is a \ch function and $T$ is a linear transformation. \quad $\triangleleft$
\end{remark}

%

\subsection{Proof of main result}\label{ss:all-put-together}

The proof makes reference to the following example of a DMIAC set that is not in (MILP-R).

\begin{example}\label{ex:dmic-is-too-big}
Consider the set $E := \{(\lambda, 2\lambda): \lambda \in \Z_+\} \cup \{(2\lambda, \lambda): \lambda \in \Z_+\}$ as illustrated in Figure~\ref{fig:counter-example}.
\begin{figure}
\begin{center}
\includegraphics[scale=.75]{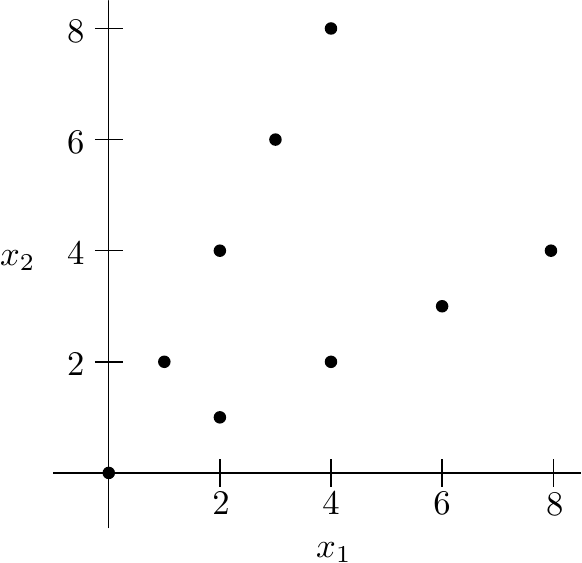}
\end{center}
\caption{The DMIAC set in Example~\ref{ex:dmic-is-too-big} that is not in (MILP-R).}
\label{fig:counter-example}
\end{figure}
This set is a DMIAC set because it can be expressed as $E =\{x \in \Z^2_+ : 2x_1 - x_2 = 0\} \cup \{x \in \Z^2_+ : x_1 - 2x_2 = 0\}.$

We  claim that $E$ is not the projection of any MILP set. Indeed, by Theorem~\ref{theorem:jeroslow-lowe} every MILP-representable set has the form \eqref{eq:jeroslow-lowe-chraracterization}. Suppose $E$ has such a form. Consider the integer points in $E$ of the form $(\lambda, 2 \lambda)$ for $\lambda \in \Z_+$. There are infinitely many such points and so cannot be captured inside of the finitely-many polytopes $P_k$ in \eqref{eq:jeroslow-lowe-chraracterization}. Thus, the ray $\lambda (1, 2)$ for $\lambda \in \Z_+$ must lie inside $\intcone \{r^1, \dots, r^t\}$. Identical reasoning implies the ray $\lambda (2,1)$ for $\lambda \in \Z_+$ must also lie inside $\intcone \{r^1, \dots, r^t\}$. But then, every conic integer combination of these two rays must lie in $E$. Observe that $(3,3) = (2,1) + (1,2)$ is one such integer combination but $(3,3) \notin E$. We  conclude that $E$ cannot be represented in the form \eqref{eq:jeroslow-lowe-chraracterization} and hence $E$ is not MILP-representable. \quad $\triangleleft$
\end{example}

We now prove the main result of the paper.

\begin{proof}[Proof of Theorem~\ref{theorem:main}]
The relationships
\begin{center}
(LP) =   (LP-R)  $\subsetneq$  (MILP)  $\subsetneq$  (MILP-R)
\end{center}
are known but we include the proof for completeness.    By Fourier-Motzkin elimination we know that  projecting variables from a system of linear inequalities gives a new system of linear inequalities so  (LP) =   (LP-R).   There are sets in  (MILP)  that are not  convex while LP sets are convex polyhedra, so (LP) $\subsetneq$ (MILP).  Since (LP) =   (LP-R),   (LP-R) $\subsetneq$ (MILP).  Since a set is always a (trivial) projection of itself, (MILP)  $\subseteq$  (MILP-R).  See Williams~\cite{williams-1} for an example of a set that is in (MILP-R) but not in (MILP).  Therefore (MILP)  $\subsetneq$  (MILP-R).

We now establish the new results

\begin{equation}\label{eq:new-results}
\textrm{(MILP-R)  =  (AC) = (AC-R) = (MIAC) =  (MIAC-R) } \subsetneq \textrm{ (DMIAC)}.
\end{equation}
We first show the equalities. We show in Theorem~\ref{theorem:milp-is-mic} that if a set  $S \in$ (MILP-R), then $S \in$ (AC) so  (MILP-R)  $\subseteq$  (AC).
Since a set is always a (trivial) projection of itself,  (AC) $\subseteq$ (AC-R). Also it trivially follows that (AC-R) $\subseteq$ (MIAC-R) and (AC) $\subseteq$ (MIAC). Finally, since a set is a projection of itself, (MIAC) $\subseteq$ (MIAC-R). Thus, we obtain the two sequences: (MILP-R) $\subseteq$  (AC) $\subseteq$ (AC-R) $\subseteq$ (MIAC-R), and (MILP-R) $\subseteq$  (AC) $\subseteq$ MIAC $\subseteq$ (MIAC-R). To complete the proof of the equalities in~\eqref{eq:new-results}, it suffices to show that (MIAC-R) $\subseteq$ (MILP-R). Consider any $S \in$ (MIAC-R). By definition, there exists a MIAC set $C \subseteq \R^n \times \R^p\times \Z^q$ such that $S = \proj_x(C)$, where we assume $x$ refers to the space in which $S$ lies in, and we let $(y,z) \in \R^p \times \Z^q$ refer to the extra variables used in the description of $C$. In Theorem~\ref{theorem:mic-is-milp}, we show that $C \in$ (MIAC) implies $C \in$ (MILP-R), i.e., there is a MILP set $C'$ in a (possibly) higher dimension such that $C=\proj_{x,y,z}(C')$. Thus, $S = \proj_x(C) = \proj_x(\proj_{x,y,z}(C'))=\proj_x(C')$. So, $S \in$ (MILP-R).
%
%
%
%
%
%

Trivially, (MIAC) is a subset of (DMAIC). From Example~\ref{ex:dmic-is-too-big} we know (MILP-R) $\subset$ (DMIAC).  Since (MILP-R) = (MIAC-R) we now have the complete proof of~\eqref{eq:new-results}.
\end{proof}


\section{Connections to consistency testers}\label{s:consistency-testers}

We now explore the conceptual connection of our main result to an established theory of consistency testers in linear and pure integer systems. Let $A$ be an $m$ by $n$ matrix. We call $V:\R^m \to \R$ an \emph{LP-consistency tester} for $A$ if for any $b\in \R^m$, $V(b) \le 0$ if and only if $\{ x \in \R^n : Ax \ge b \}$ is nonempty. We call $F : \R^m \to \R$ an \emph{IP-consistency tester} for $A$ if for any $b \in \R^m$, $F(b) \le 0$ if and only if $\{ z \in \Z^n : Az \ge b \}$ is nonempty. The following result shows that FM elimination is a source of LP-consistency testers.

\begin{theorem}[Corollary~2.11 in Martin~\cite{martin1999large}]\label{theorem:FM-gives-consistency-tester}
Let $u^1, \dots, u^t$ be the FM multipliers of the matrix $A$ from eliminating all $x$ variables in the system $Ax \ge b$. Then $U(b) = \max_{i = 1, \dots, t} u^i b$ is an LP-consistency tester for $A$.
\end{theorem}

Any LP-consistency tester that arises from applying Fourier-Motzkin to the matrix $A$ is called a \emph{FM-based LP-consistency tester}. For a given matrix $A$ there can be more than one FM-based LP-consistency tester. The FM elimination procedure has two flexibilities that can be adjusted in a given implementation:

\begin{enumerate}[label=(F\arabic*), align=left]
    \item \emph{(Scaling)} Differing nonnegative scalings of the rows of the matrix in the process of eliminating a column. For instance, a common implementation is to first normalize the coefficients in the column to be eliminated to be $\pm 1$. \label{item:flexibility-scaling}
    \item \emph{(Ordering)} Different orders of eliminating columns. For instance, one could eliminate the first column, followed by the second, etc. Alternatively one could eliminate the last column, second-to-last, etc. \label{item:flexiblity-order}
\end{enumerate}

Different choices of scaling and ordering gives rise to different sets of inequalities involving $b$ and hence different consistency testers. However, all FM-based LP-consistency testers share some common properties. We call the cone $C_P = \{ u \in \R^m : uA = 0, u \ge 0 \}$ the \emph{projection cone} of $A$. The FM multipliers have the following relationship with $C_P$.

\begin{theorem}[Proposition~2.3 in Martin~\cite{martin1999large}]
The extreme rays of the projection cone $C_P$ are contained in the set $\left\{u^1, \dots, u^t\right\}$ of FM multipliers of matrix $A$.
\end{theorem}

This connection grounds the following result.

\begin{theorem}\label{theorem:extreme-rays-always-there}
Let $e^1, \dots, e^r$ denote a set of extreme rays of $C_P$ and set $E(b) = \max_{k=1, \dots, r} e^k b$. Then
\begin{enumerate}[label=(\roman*)]
\item $E$ is an LP-consistency tester, and
\item every LP-consistency tester $V$ of the form $V(b) = \max_{i \in I} v^i b$ where $I$ is a finite index set and $v^i \in \R^m$ is such that the set $\left\{v^i : i \in I \right\}$ contains a positive multiple of each extreme ray $e^k$.
\end{enumerate}
\end{theorem}
\begin{proof}
(i) To show $E$ is a consistency tester, first suppose $Ax \ge b$ is consistent. Then $e^k A x \ge e^k b$ is also feasible since $e^k \ge 0$ and since $e^kA=0$ this implies $0 \ge e^k b$. Since this is true for all $k$ we have $E(b) \le 0$. Next, suppose $Ax \ge b$ is inconsistent. Then by Farkas Lemma there exists a $u \in C_P$ such that $uA = 0$, $u\ge 0$ and $ub > 0$. Since $u$ is a conic combination of the $e^k$, this means there exists a $k$ such that $e^k b > 0$. Hence, $E(b) > 0$. This implies $E$ is a consistency tester.

(ii) Let $V$ be a consistency tester and let $e$ be an arbitrary extreme ray of $C_P$. Let $J = \left\{j : e_j > 0\right\}$ denote the support of $e$. We make the following two claims, whose proofs are straightforward.

\begin{claim}\label{claim:strict-contain-supports}
The support of any of the $v^i$ cannot be a strict subset of the support of $e$. In particular, if $\{j : v^i_j > 0 \}$ is a subset of $J$ then it must equal $J$.
\end{claim}

\begin{claim}\label{claim:what-happens-when}
If $\{j : v^i_j > 0\} = J$ then $v^i$ is a positive multiple of $e$.
\end{claim}

Now, consider the right-hand side $\bar b$ where $\bar b_j = 1$ for all $j \in J$ and $\bar b_j = -M$ for all $j \notin J$ where $M$ is an arbitrarily large real number.  Since $e\bar b > 0$,  the system $Ax \ge \bar b$ is not feasible.  Then, since $V$ is an LP-consistency tester, there exists an $i$ such that $v^i \bar b > 0$.  But since $M$ is arbitrarily large it must be the case that $v^i_j = 0$ for all $j \notin J$. But this implies that the support of $v^i$ is contained in the support of $J$, and so its support must be exactly $J$ by Claim~\ref{claim:strict-contain-supports}. Then Claim~\ref{claim:what-happens-when} implies $v^i$ is a positive multiple of $e$.

Hence, every extreme ray is a positive multiple of some $v^i$ and so the theorem is proved.
\end{proof}

One interpretation of the previous result is that a set of extreme rays of $C_P$ forms a \emph{minimal} LP-consistency tester for $A$. For this reason, for LP-consistency tester $V(b) = \max_{i \in I} v^i b$ $v^i$ if vector $v^i$ is \emph{not} an extreme ray of $C_P$ we call it  \emph{redundant}. Typically, FM-based LP-consistency involve many redundant vectors (although see Example~\ref{ex:ipco-example-defend-against-flexibilities} below). A key idea of this section is that although these vectors are redundant for an LP-consistency tester, they may not be for an associated IP-consistency testers. Our next task is to make this statement precise.

{\bf One interpretation of Theorem~\ref{theorem:b-j-mod} is that there exists an IP-consistency tester of the form $F(b) = \max_{i \in I} f_i(b)$ where $f_i$ for $i \in I$ is a finite collection of \ch functions.} Our goal is to connect IP-consistency testers of this type to FM-based LP-consistency testers. To do so we need the following definitions and observations.

\begin{definition}\label{definition:carrier}
The \emph{carrier} of a \ch function $f: \R^n \to \R$, denoted $\carr(f)$, is the linear function $g$ that results when all ceiling operators in $f$ are removed. For example, if $f(x_1, x_2) = \lceil \lceil x_1 + x_2 \rceil + 3x_2 \rceil + x_1$ then $\carr(f) = 2x_1 + 4x_2$. \quad $\triangleleft$
\end{definition}

For a more precise definition of carrier see Definition~2.9 in Blair and Jeroslow \cite{blair82}, although this level of formality is not needed for our development. An important fact is that the carrier of the \ch function is unique.

Related to the concept of the carrier is the reverse operation, taking a linear function and turning it into a \ch function through the use of ceiling operations.

\begin{definition}\label{definition:ceilingization}
A \emph{ceilingization} of a linear function $g$, denoted $\ceil g$, is any \ch function $f$ such that $\carr(f) = g$.
\end{definition}

The ceilingization of a linear function need not be unique. Indeed, we have two types of (related) flexibilities.

\begin{enumerate}[label=(F\arabic*), align=left]
\setcounter{enumi}{2}
\item \emph{(Ceiling pattern)} Given a linear function $g$, ceilings can be inserted to include just certain variables, certain terms, or across terms. For instance, $\lceil 2x_1 + 4x_2 \rceil$, $\lceil 2x_1 \rceil + 4x_2$, and $\lceil 2x_1 \rceil + \lceil 4x_2 \rceil$ are all ceilingizations of $g(x_1,x_2) = 2x_1 + 4x_2$.     \label{item:ceiling-pattern}
\item \emph{(Break-ups)} New terms can be created by ``breaking up'' terms and inserting ceilings within the newly created terms. For instance, $\lceil x_1 \rceil + \lceil x_1 \rceil + 4x_2$ and $\lceil \tfrac{1}{2}x_1 \rceil + \lceil \tfrac{1}{2}x_1 \rceil + \lceil x_1 \rceil + 4x_2$ are both ceilingizations of $g(x_1,x_2) = 2x_1 + 4x_2$.        \label{item:breakups}
\end{enumerate}

The following result builds a connection between LP-consistency testers and IP-consistency testers through the lens of carriers.

\begin{theorem}[Theorem~5.20 in Blair and Jeroslow~\cite{blair82}]\label{theorem:ip-gives-rise-to-lp}
If $F(b) = \max_{i \in I} f_i(b)$ is an IP-consistency tester for
$A$ where the $f_i$ are \ch functions and $I$ is finite, then $G(b) = \max_{i \in I} g_i(b)$ is an LP-consistency tester for $A$, where $g_i = \carr(f_i)$ for $i \in I$.
\end{theorem}

In other words, given an IP-consistency tester based on \ch functions, it is a simple matter to produce an LP-consistency tester -- just erase all the ceilings! However, this raises the question of a potential converse.

\begin{question}\label{question:lp-to-ip}
Given an LP-consistency
tester $G(b) = \max_{i \in I} g_i(b)$ where $g_i$ are linear for all $i$ and $I$ is finite,
does there exist a ceilingization $f_i$ of the $g_i$ for all $i$ such that $F(b) = \max_{i \in I} f_i(b)$
is an IP-consistency tester?	
\end{question}

For brevity, we abuse terminology and call $F(b) = \max_{i \in I} f_i(b)$ a \emph{ceilingization} of $G(b) = \max_{i \in I} g_i(b)$ if each $f_i$ is a ceilingization of $g_i$. Then, we can rephrase the question as whether there always exists a ceilingization of an LP-consistency that is an IP-consistency tester.

The answer to this question is ``no'', as illustrated in the following example.

\begin{example}\label{ex:ipco-example}
Consider the linear system
\begin{align}\label{eq:fm-example}
\begin{array}{rrrcrr}
- x_1  &+\frac{1}{2} x_2 &-\frac{1}{10} x_3& \ge & b_1 & \\
x_1 &-\frac{1}{4} x_2&&\ge& b_2& \\
&-x_2&+x_3&\ge& b_3& \\
&&x_3&\ge& b_4& \\
&&-x_3&\ge& b_5& .
\end{array}
\end{align}
We generate an LP-consistency tester $G$ using FM elimination (and Theorem~\ref{theorem:FM-gives-consistency-tester}). FM elimination yields
\begin{align}
      0 &\ge  2b_1 + 2 b_2 + \tfrac{1}{2} b_3 + \tfrac{3}{10} b_5 \label{eq:fm-ray1} \\
      0 & \ge \tfrac{1}{10}b_4 + \tfrac{1}{10}b_5. \label{eq:fm-ray2}
\end{align}
when eliminating the variables in the order $x_1$, $x_2$, then $x_3$. This yields the LP-consistency tester
\begin{equation}\label{eq:example-lp-tester}
G(b_1, b_2,b_3,b_4, b_5) = \max \left\{2b_1 + 2b_2 + \tfrac{1}{2}b_3 + \tfrac{3}{10} b_5, \tfrac{1}{10}b_4 + \tfrac{1}{10}b_5 \right\}
\end{equation}
We now show that there is no possible ceilingization of $G$ that yields an IP consistency tester.
Let $\mathcal{B}$ denote the set of all $b = (b_1, \ldots, b_5) \in \R^5$ such that there exist $x_1, x_2, x_3 \in \Z$ satisfying system \eqref{eq:fm-example}. In particular, $b^1 = (0,0,0,1,-1) \notin \mathcal{B}$ while $b^2 = (-1, 0,0,1,-1) \in \mathcal{B}$. Consider $b^1$. This forces $x_3 = 1$  and the only feasible values for $x_1$ are $ 1/10 \le x_1 \le 4/10.$  Therefore, for this set of $b$ values applying the ceiling operator to some combination of terms in~\eqref{eq:fm-ray1}-~\eqref{eq:fm-ray2}  must result in either~\eqref{eq:fm-ray1} positive or~\eqref{eq:fm-ray2}   positive.  Since $b^1_1 = b^1_2 = b^1_3  = 0$  and $b^1_5 = b^2_5 = -1$ there is no ceiling operator that can be applied to any term in~\eqref{eq:fm-ray1} to make  the right hand side positive.  Hence a ceiling operator must be applied to~\eqref{eq:fm-ray2} in order to make the right hand side positive for $b^1_4 = 1$ and $b^2_5 = -1.$  However, consider $b^2$. For this right-hand-side, $x_{1}  = x_{2} = x_{3} = 1$ is feasible.  Since we still have  $b^2_4 = 1$ and $b^2_5 = -1$, the ceiling operator applied to~\eqref{eq:fm-ray2} will incorrectly conclude that there is no integer solution with right-hand side $b^2$. \quad $\triangleleft$
\end{example}

In this example the LP-consistency tester \eqref{eq:example-lp-tester} is minimal in the sense of Theorem~\ref{theorem:extreme-rays-always-there} -- it has only two linear terms in the LP-consistency tester and both are extreme rays of the projection cone. This suggests that although redundant vectors are not needed for the linear cases, they may be needed in the integer case. Since FM elimination is typically a source of redundant vectors, the next question refines Question~\ref{question:lp-to-ip} in this context.

\begin{question}\label{question:fm-to-ip}
Given a matrix $A$, does there exist an FM-based LP-consistency tester $G$ for $A$ such that there exists a ceilingization $F$ of $G$ that is an IP-consistency tester for $A$.
\end{question}

To our knowledge, this question is open. Indeed, it seems hard to answer because of the inherent flexibilities in deriving FM-based LP-consistency testers and in ceilingizing affine functions -- (F1)--(F4) provide four sources of flexibility that can be exploited in deriving a $G$ and $F$ to answer Question~\ref{question:fm-to-ip} positively. The following examples show the power of this flexibility, but also its limitations.

\begin{example}[Example~\ref{ex:ipco-example}, continued]\label{ex:ipco-example-defend-against-flexibilities}

We already showed in Example~\ref{ex:ipco-example} that there is no ceilingization of LP-consistency tester \eqref{eq:example-lp-tester} that yields an IP consistency tester. We demonstrated this by showing every resulting ceilingization cannot separate $b^1$ and $b^2$, while $b^1$ is not a feasible right-hand side and $b^2$ is. In other words, flexibilities \ref{item:ceiling-pattern} and \ref{item:breakups} are not sufficient, given a particular FM-based LP-consistency tester.

However, \eqref{eq:example-lp-tester} is not the only FM-based LP-consistency tester possible.
%
%
We leverage flexibility (F2) and eliminate the variables in a different order:
eliminate
$x_2$ first, followed by $x_3$ then $x_1$ to yield the following:
\begin{align*}
      0 & \ge  4b_1 + 4b_2 + b_3 + \tfrac{1}{5} b_4 + \tfrac{4}{5} b_5 \\
      0 & \ge  \tfrac{16}{3} b_1 + \tfrac{16}{3}b_2 + 4 b_3 + \tfrac{4}{5} b_5 \\
      0 & \ge  b_4 + b_5.
\end{align*}
Observe that there is a simple ceilingization that can separate $b^1$ and $b^2$. Simply round the top inequality to $4b_1 + 4b_2 + b_3 + \lceil \tfrac{1}{5} b_4 \rceil + \lceil \tfrac{4}{5} b_5 \rceil$. Indeed, evaluated at $b^1 = (0,0,0,1,-1)$ this ceilingized inequality evaluations to $0 \not \ge \lceil \tfrac{1}{5} (1) \rceil + \lceil \tfrac{4}{5} (-1) \rceil = 1$. It is straightforward to see that $b^2$ is still feasible. This overcomes the deficiency discussed in Example~\ref{ex:ipco-example}.  Observe that in the above three inequalities involving the $b_i$, one is redundant, in the sense of not being a conic combination of the other two.

Another way to approach this example is to simply add a redundant inequality $x_1  \ge  b_1 + 2 b_2 + \frac{1}{10} b_4$ to the original system~\eqref{eq:fm-example}. Integrality of $x_1$ implies  $ x_1  \ge  \lceil b_1 + 2 b_2 + \frac{1}{10} b_4 \rceil.$
 Applying Fourier-Motzkin elimination to~\eqref{eq:fm-example} along with  $ x_1  \ge  \lceil b_1 + 2 b_2 + \frac{1}{10} b_4 \rceil$  generates  the additional inequality $0   \ge b_1 + \frac{1}{2} b_3 + \lceil b_1 + 2 b_2 + \frac{1}{10} b_4 \rceil + \frac{4}{10} b_5$, which separates $b^1$ and $b^2$. The idea of adding redundant constraints is central to our method in Section~\ref{s:MILP-as-MIC}. \quad $\triangleleft$
\end{example}

The previous example leaves open the question of whether changing the order of elimination in the FM procedure gives rise to a consistency tester for the corresponding integer program. The next example shows that changing the order may be insufficient given a particular scheme of ceilingization.

\begin{example}\label{ex:amitabh-example} Consider the linear system
\begin{align}\label{eq:fm-example-2}
\begin{array}{rrcrr}
3x_1  &+ 2x_2 & \ge & b_1 &\\
-3x_1 &-2 x_2&\ge& b_2 &\\
3x_1 &-2x_2&\ge& b_3 &\\
-3x_1& + 2x_2&\ge& b_4 &.
\end{array}
\end{align}
We will now apply Fourier-Motzkin on this system with a very simple ceilingization rule: whenever we derive a constraint with integer coefficients on the left hand side, we put a ceiling operator on the right hand side. Fourier-Motzkin will be applied with the canonical scalings (see (F1)) where the variable being eliminated has coefficients $\pm1$. Moreover, we will apply the procedure under all possible variable orderings. Under both orderings, we will see that the \ch inequalities obtained do not give a consistency tester for the integer program. This will show that the flexibility of (F2) alone is not enough.

Under the ordering $x_1, x_2$, the final \ch inequalities obtained under the above scheme are
\begin{equation*}
 0 \geq \frac{\lceil b_1 \rceil + \lceil b_2 \rceil}{3}, \;\;\;\; 0 \geq \frac{\lceil b_3 \rceil + \lceil b_4 \rceil}{3}, \;\;\;\; 0 \geq \bigg\lceil\frac{\lceil b_1 \rceil + \lceil b_4 \rceil}{4}\bigg\rceil + \bigg\lceil\frac{\lceil b_2 \rceil + \lceil b_3 \rceil}{4}\bigg\rceil
\end{equation*}

Under the ordering $x_2, x_1$, the final \ch inequalities obtained under the above scheme are
\begin{equation*}
0 \geq \frac{\lceil b_1 \rceil + \lceil b_2 \rceil}{2}, \;\;\;\; 0 \geq \frac{\lceil b_3 \rceil + \lceil b_4 \rceil}{2}, \;\;\;\; 0 \geq \bigg\lceil\frac{\lceil b_1 \rceil + \lceil b_3 \rceil}{6}\bigg\rceil + \bigg\lceil\frac{\lceil b_2 \rceil + \lceil b_4 \rceil}{6}\bigg\rceil
\end{equation*}
Neither of the above two systems give a consistency tester for the integer feasibility problem for~\eqref{eq:fm-example-2}. This is because setting $b_1 = 1, b_2 = -4, b_3 = -1, b_4 = -2$ satisfies all the \ch inequalities above. However, the polyhedron obtained with these right hand sides in~\eqref{eq:fm-example-2} does not contain any integer points in $\Z^2$. \quad $\triangleleft$
\end{example}

Although an answer to Question~\ref{question:fm-to-ip} as stated seems elusive, the theory discussed above (particularly, the part that builds on the approach of Schrijver~\cite{schrijver86}) provides a positive answer to an adjusted question. The idea is to add redundant constraints to the initial system $Ax \ge b$ in order to generate even more redundant vectors in the resulting FM-based LP-consistency tester. In other words, although the FM procedure does generate redundant vectors from the original system, even this level of redundancy is insufficient to produce an IP-consistency tester through ceilingization. However, our results from Section~\ref{s:MILP-as-MIC} do provide a level of ``redundancy'' that does suffice. This insight is captured in the next result.

Let $v^kAx \ge v^k b$ for $k = 1, \dots, K$ be a collection of redundant inequalities
to the linear system $Ax \ge b$ where the $v^k$ are independent of $b$. Let $A'$ be the
matrix $A$ with appended rows $v^kA$ for $k = 1,\dots, K$. Let $b'(b) = (b, v^k b : k = 1, \dots, K)^\top$, (that is, if we think of $b'$ as a function of $b$, appending values to the bottom of $b$). Let $u^1, \dots, u^t$ denote a set the FM multipliers for the matrix $A'$.

\begin{theorem}\label{theorem:we-can-append-to-get-it}
There exists a choice of row multipliers $v^k$ and FM multipliers $u^i$ (as described above) such that there exists a ceilingization of $G(b) = \max_{t = 1, \dots, t} u^t b'(b)$ that is an IP-consistency tester.
\end{theorem}
\begin{proof}
This is a consequence of the procedure described in the proof of Theorem~\ref{theorem:b-j-mod} (which follows Theorem~\ref{theorem:schrijver-23.4}). Observe that the system $\left\{x : Ax \ge b\right\}_I$ is a system of the form $A'x \ge b'$ with appropriate rounding of the right-hand sides. Then, FM elimination produces an IP-consistency tester (this is precisely the conclusion of Theorem~\ref{theorem:b-j-mod}).
\end{proof}

This result provides a perspective on our results in the previous section, which builds on the work in Schrijver~\cite{schrijver86}. The theory of \ch closures provides the ``right'' redundant constraints to add to the system, and a method to ceilingize the resulting right-hand sides, to recover an IP-consistency tester.


\section{Connections to variable elimination schemes}\label{s:variable-elimination}

In the last section we saw the power of FM elimination for producing consistency testers for linear and some of its potential limitations for producing consistency testers for integer programs. Extending the FM elimination procedure to handle the elimination of \emph{integer} variables has been a goal of-repr several researchers in past decades. One benefit of this exploration is the possibility of producing IP-consistency testers. Other benefits include solving integer programs and understanding notions of duality for integer systems (for more details see \cite{williams86,williams-hooker,balas2011projecting}). This section explores some implications of our methodology for the topic of elimination of integer variables.

To carefully describe what we mean by a \emph{variable elimination scheme} (VES) we first describe the elimination of a single variable. A VES takes a description of a mixed integer set $S \subseteq \R^n \times \Z^q$ involving affine \ch functions and algorithmically produces a representation of the projection $\proj_{x_{-j},z} S$ for some $j \in \{ 1,2, \dots, n \}$ or $\proj_{x, z_{-k}} S$ for some $k \in \{ 1, 2, \dots, q \}$, again using only affine \ch functions. We are a bit vague when we say a description of a set ``involving affine \ch functions''. We allow this to include both MIAC sets and DMIAC sets, as defined in Section~\ref{s:preliminaries}. We also restrict attention to elimination schemes that ``specialize'' to FM elimination when eliminating a continuous variable $x_j$ that does not appear in any ceiling operations. Indeed, in this case, it is straightforward to see that FM suffices to recover the projection.

Next, we describe how a VES approaches the projection of more than one variable from the set $S$. A VES, like FM elimination, will attack this \emph{sequentially}. For instance, if we want to find a description of the projection $\proj_{x_{-j}, z_{-k}} S$ for some $j \in \{ 1,2, \dots, n \}$ and $k \in \{ 1, 2, \dots, q \}$, a VES will first eliminate $x_{-j}$ (or $z_{-k}$) to produce $\proj_{x_{-j},z} S$ (or $\proj_{x, z_{-k}}$). Then, the next step is to eliminate the remaining variable from the description of the projected set.

The existing literature focuses on variable elimination schemes (VES's) for the special case where the starting set is the set of pure integer points inside a polyhedron and the output is a DMIAC set (after eliminating more than the first variable). The VES of Williams and Hooker \cite{williams-hooker} is described in some detail in Section~\ref{ss:williams-hooker} below.

Our approach (the focus of Section~\ref{ss:lift-and-project}) complements the existing methods along two important directions. First, our method starts with an arbitrary MIAC set, not only mixed integer polyhedral sets. Second, we are guaranteed to output a MIAC set, not just a DMIAC set. Also in Section~\ref{ss:williams-hooker} we show that DMIAC sets are not necessarily MILP-R sets. Hence, maintaining a MIAC description after projection is critical to our characterization result of MILP-R sets as MIAC sets (see Example~\ref{ex:dmic-is-too-big} below).

In a related direction, Ryan shows (see Theorem 1 in~\cite{ryan1991}) that $Y$ is a finitely generated integral monoid if and only if there exist \ch functions $f_{1}, \ldots, f_p$ such that $Y = \{ b  \, : \, f_{i}(b) \leq 0, \, i = 1, \ldots, p\}.$ A finitely generated integral monoid $Y$ is MILP representable since, by definition, there exists a matrix $A$ such that $Y = \{  b \, :  \, b = Ax,  x \in \Z^{n}_{+}\}$.  Thus, an alternate proof of Ryan's characterization follows from Theorems~\ref{theorem:mic-is-milp} and~\ref{theorem:milp-is-mic} and Remark~\ref{rem:homogeneous}.

Ryan~\cite{ryan1991} further states that  ``It is an interesting open problem to find an elimination scheme to construct the \ch constraints for an arbitrary finitely generated integral monoid.'' We interpret this statement as asking for a VES as defined at the outset of this section.  Ryan was aware of the methodology of Williams in \cite{williams-1} and this method fell short of her goal. In Section~\ref{ss:lift-and-project} we provide an approach that positively answers the conjecture of Ryan. In fact, it answers positively the more general question: does there exist a VES that provides a MIAC representation of a MILP-R set (which is guaranteed to exist by Theorem~\ref{theorem:milp-is-mic})?

\subsection{Eliminating a single variable, and the difficulty of eliminating subsequent variables}\label{ss:fmpr}

It turns out that eliminating the first integer variable from a linear system can be easily granted by a simple extension of FM elimination. Consider the following procedure. For simplicity, assume that the first variable to be eliminated is $x_1$. Given a linear system $Ax \ge b$ where $x \in \mathbb R^n$ and $A = (a_{ij})$ and let
\begin{align*}
\mathcal H_+ & := \left\{i : a_{i1} > 0 \right\}  \\
\mathcal H_- & := \left\{i : a_{i1} < 0 \right\}  \\
\mathcal H_0 & := \left\{i : a_{i1} = 0 \right\}
\end{align*}
We will assume that $\mathcal H_+$ and $\mathcal H_-$ are both nonempty and hence $x_1$ can be eliminated. The case where one of $\mathcal H_+$ or $\mathcal H_-$ being empty means the problem is unbounded in $x_1$ and the integer projection in this case is straightforward. The case where both of $\mathcal H_+$ or $\mathcal H_-$ are empty means that $x_1$ does not appear in the system, a case that we ignore.

FM elimination stems from that fact that if $Ax \ge b$ where $x = (x_2, \dots, x_n)$ then
\begin{equation}\label{eq:bounds-on-x1}
\tfrac{b_p}{a_{p1}} - \sum_{j=2}^n \tfrac{a_{pj}}{a_{p1}} x_j \le x_1 \le \tfrac{b_q}{a_{q1}} - \sum_{j=2}^n \tfrac{a_{qj}}{a_{q1}} x_j
\end{equation}
for all $p \in \mathcal H_+$ and $q \in \mathcal H_-$. Conversely, if any choice of variables $x_2, x_3, \dots, x_n$ satisfies
\begin{align}
\sum_{j=2}^n a_{ij}x_j & \ge b_i    && \text{ for } i \in \mathcal H_0 \label{eq:fm-standard1}\\
\sum_{j=2}^n \left(\tfrac{a_{pj}}{a_{p1}} - \tfrac{a_{qj}}{a_{q1}} \right)x_j &\ge \tfrac{b_p}{a_{p1}} -  \tfrac{b_q}{a_{q1}} && \text{ for } p \in \mathcal H_+ \text{ and } q \in \mathcal H_-, \label{eq:fm-standard2}
\end{align}
there exists a choice of $x_1$ that satisfies \eqref{eq:bounds-on-x1}, resulting in $Ax \ge b$ where $x = (x_1, x_2, \dots,x_n)$. In other words, \eqref{eq:fm-standard1}--\eqref{eq:fm-standard2} characterizes $\proj_{x_{-1}} \left\{x \in \R^n : Ax \ge b \right\}$.

However, this does not characterize the projection of the \emph{integer} values of $x_1$. All integer $x_1$ satisfy \eqref{eq:bounds-on-x1} but the converse may not be true. There is a simple fix. Introduce ceilings as follows:
\begin{equation}\label{eq:bounds-on-x1-integer}
\left\lceil \tfrac{b_p}{a_{p1}} - \sum_{j=2}^n \tfrac{a_{pj}}{a_{p1}} x_j \right \rceil \le x_1 \le \tfrac{b_q}{a_{q1}} - \sum_{j=2}^n \tfrac{a_{qj}}{a_{q1}} x_j
\end{equation}
and no additional integer values for $x_1$ can be introduced. This is proven formally in the following result.

\begin{prop}\label{prop:project-out-one-variable}
The set $\proj_{x_{-1}} \left\{x \in \Z^n : Ax \ge b \right\}$ equal all integer vectors $(x_2, \dots, x_n)$ such that
\begin{align}
\sum_{j=2}^n a_{ij}x_j & \ge b_i    && \text{ for } i \in \mathcal H_0 \label{eq:nothing-happens}\\
\tfrac{b_q}{a_{q1}} - \sum_{j=2}^n \tfrac{a_{qj}}{a_{q1}}x_j - \left\lceil \tfrac{b_p}{a_{p1}} - \sum_{j=2}^n \tfrac{a_{pj}}{a_{p1}} x_j \right\rceil &\ge  0 && \text{ for } p \in \mathcal H_+ \text{ and } q \in \mathcal H_- \label{eq:what-to-do-now}
\end{align}
\end{prop}
\begin{proof}
Let $(\bar x_2, \dots, \bar x_n) \in \proj_{x_{-1}} \left\{x \in \Z^n : Ax \ge b \right\}$. There exists exists an integer $\bar x_1$ such that $A\bar x \ge b$ where $\bar x = (x_1, x_2, \dots, x_n)$. Hence, it must be that $(\bar x_2, \dots, \bar x_n) \in \proj_{x_{-1}} \left\{x \in \R^n : Ax \ge b \right\}$ and so \eqref{eq:bounds-on-x1} must be satisfied by $x_1$. Since $x_1$ is integer, it must also satisfy \eqref{eq:bounds-on-x1-integer} when we round up the right-hand side of \eqref{eq:bounds-on-x1}. Hence, $\bar x$ satisfies the system of equations \eqref{eq:nothing-happens}--\eqref{eq:what-to-do-now}.

Conversely, suppose $(\bar x_2, \dots, \bar x_n)$ are integers that satisfy \eqref{eq:nothing-happens}--\eqref{eq:what-to-do-now}. Set $\bar x_1 = \left\lceil \tfrac{b_p}{a_{p1}} - \sum_{j=2}^n \tfrac{a_{pj}}{a_{p1}} \bar x_j \right \rceil$. Clearly, this choice of $\bar x_1$ satisfies \eqref{eq:bounds-on-x1} and is integer. Hence, $\bar x = (\bar x_1, \bar x_2, \dots, \bar x_n)$ is integer and satisfies $A\bar x \ge b$ and thus
\begin{equation*}
(\bar x_2, \dots, \bar x_n) \in \proj_{x_{-1}} \left\{x \in \Z^n : Ax \ge b \right\},
\end{equation*}
as required.
\end{proof}

One would like to continue in this way to build a VES that sequentially eliminate variables. But there is one key challenge: FM elimination works only on linear systems but \eqref{eq:nothing-happens}--\eqref{eq:what-to-do-now} have ceilings!

This ``ceiling quagmire'' calls for new ideas. One approach is to introduce disjunctions. This technique is described at a high level in the next section. However, as we will see there, introducing disjunctions moves us outside the class of MIAC sets that characterize MILP representability. Our approach is to ``lift'' the formulation by introducing new integer variables as in Theorem~\ref{theorem:mic-is-milp} and then ``project'' using the technique described in Theorem~\ref{theorem:milp-is-mic}. This ``lift-and-project'' method is described in detail in Section~\ref{ss:lift-and-project}.

\subsection{The Williams-Hooker elimination scheme}\label{ss:williams-hooker}

In this section we briefly describe the main idea and some the implications of the elimination scheme of Williams and Hooker \cite{williams-hooker}.  For short (and in parallel to FM elimination), we call their procedure WH elimination.  WH elimination builds on the previous work of Williams in \cite{williams-2,williams86,williams-1}.

WH elimination is a VES that takes as input a polyhedral description of a set of mixed integer vectors in the form of linear inequalities. Variables are eliminated in a similar manner as FM elimination with an additional step of accounting for integrality. This accounting introduces two additional mathematical features not present in FM elimination: congruence relations and disjunctions. The congruence relation relates to the coefficients on the variables that are eliminated and the disjunctions correspond to an exhaustive enumeration of congruence classes, e.g. $0 \bmod 3$, $1 \bmod 3$, and $2 \bmod 3$. These new mathematical features get around the ``quagmire'' described at the end of Section~\ref{ss:fmpr}.

WH elimination is a powerful technique that can be used to analyze a variety of integer programming-related questions. For our specific context, it can be used to establish the following result.

\begin{theorem}\label{theorem:w-h-gives-projection}
Every MILP-R set is a DMIAC set. That is, $(\text{MILP-R}) \subseteq (\text{DMIAC})$.
\end{theorem}

Theorem~\ref{theorem:w-h-gives-projection} is not explicitly stated in \cite{williams-2,williams-1,williams-hooker} but it is a direct consequence of their method.
We established this containment already in Theorem~\ref{theorem:main} using a different methodology, first of all showing the equivalence between (MILP-R) and (MIAC). 
Example~\ref{ex:dmic-is-too-big} shows that the converse is not true. 

\subsection{A lift-and-project variable elimination scheme}\label{ss:lift-and-project}

We now describe a VES that takes as input an arbitrary MIAC set. We describe only a single variable elimination step. Since it takes as input an arbitrary MIAC set and produces as output its projection described as a MIAC set, it can be used iteratively to sequentially elimination all variables.

\noindent \hrulefill
\begin{center}
\textsc{Lift-and-project method for eliminating a single variable}
\end{center}

\noindent \textsc{Input}:  Mixed integer set $S$ described by a system of affine \ch inequalities $\{ (x,z) \in \R^n \times \Z^q : f_i(x,z) \le b_i \text{ for } i = 1, \dots, m\}$ and a variable to project, either $x_j$ or $z_j$.
\vskip 5pt
\noindent \textsc{Output}: A system of affine \ch inequalities describing the projection of $S$ onto all but one of its variables; that is, $\proj_{x_{-j},z} S$ or $\proj_{x,z_{-j}} S$.
\vskip 5pt
\noindent \textsc{Procedure}:
\vskip 5pt
\begin{enumerate}
\item If the variable to project is $x_j$ AND $x_j$ does not appear in any of the ceiling operators of any of the $f_i$ then use FM elimination to eliminate $x_j$ and return the resulting system. Else, go to 2.
\item \emph{Lift step.} Introduce integer variables $w_k$ for each ceiling function that involves $x_j$ (alternatively $z_j$) following the procedure described in Lemma~\ref{lem:reduce-ceil-count}. Suppose $K$ integer variables are introduced and set $T$ (with total ceiling count 0) denotes the resulting MILP set in $\R^n \times \Z^q \times \Z^K$. 
\item \emph{Project step.} Use the procedure described in Theorem~\ref{theorem:milp-is-mic} to find eliminate variables $w_1, \dots, w_K$ and $x_j$ (alternatively $z_j$) to return the resulting characterization of $\proj_{x_{-j},z} S$ (alternatively $\proj_{x, z_{-j}} S$).
\end{enumerate}
\noindent \hrulefill
\vskip 5pt
Clearly, the resulting algorithm of sequentially applying the above procedure produces a variable elimination scheme. The lifting into higher dimensions overcomes the ``quagmire'' discussed at the end of Section~\ref{ss:fmpr}. Moreover, eliminating the lifted variables $w_1, \dots, w_K$ in the projection step of the procedure produces (potentially many) redundant inequalities to the description of $S$ in its original variable space. As discussed in Section~\ref{s:consistency-testers}, these additional redundant constraints are useful in describing the integer projection. This procedure positively answers the question of Ryan~\cite{ryan1991} and provides a projection algorithm in a similar vein to Williams \cite{williams-1,williams-2,williams-hooker}  and Balas \cite{balas2011projecting} but {\it without} the use of disjunctions.

\section{Conclusion}\label{s:conclusion}

This paper describes a novel hierarchy  of linear representable sets, mixed-integer linear representable sets  and sets represented  by affine \ch functions.   This hierarchy is summarized in our main result (Theorem~\ref{theorem:main}). Our results show that affine \ch functions  are a unifying tool for mixed-integer linear optimization, incorporating both integrality and the notion of projection. We then explore a variety of implications of this hierarchy. For instance, we   extend and contextualize the theory of consistency testers for integer programs, which has traditionally used the tool of \ch functions, to the more general setting of MILP-R sets. Moreover, we provide a new variable elimination scheme for studying MILP-representable systems that builds on the existing literature, which was based on a combination of disjunctions and ceiling operations for pure polyhedral integer systems.

We are intrigued by the possibility that our results could be used in applications.  The use of  AC sets could provide an opportunity for modeling, as the operation of rounding affine inequalities has an inherent logic that may be understandable for particular applications. If a problem can be modeled using AC constraints, then we know it has an mixed-integer linear representation in some higher dimension.  We leave the full exploration of this issue as an area for future research. Here, we provide an illustrative example to underscore this point. 

\begin{example}
Consider a production batch-size problem. A product can be either not be produced, or if we produce a {\it positive} quantity we must produce between 50 and 200 units. In other words

$$
x = 0  \qquad {\rm OR}  \qquad  50  \le   x \le 200.
$$

Based on the results in this paper there are three {\it equivalent representations.}

\noindent {\bf Representation 1: Disjunctive representation}  $P_1 \cup P_2$  where $P_1 = \{  x |   x = 0 \}$  and $P_2 = \{  x |    50  \le   x \le 20 \}.$

\noindent {\bf Representation 2:  MILP-R set}

The following standard formulation introduces the auxiliary binary variable $y$:
\begin{eqnarray*}
x &\ge&  50 y \\
x &\le&  200 y \\
y &\in&  \{ 0, 1 \}.
\end{eqnarray*}

\noindent {\bf Representation 3:  AC  set}

The AC constraint
\begin{equation}\label{eq:this-one}
x/200  + \lceil -1/50 x \rceil \le 0 
\end{equation}
admits the zero solution and all solutions in the closed interval $[50, 200]$, that is all solutions in $P_1 \cup P_2.$  Also, strictly negative values of $x$ and values of $x$ in the open interval $(0, 50)$ are not feasible to \eqref{eq:this-one}.   However, \eqref{eq:this-one} does admit values of $x$ greater than $200$ such as $201$.  Hence we add
\begin{equation}\label{eq:that-one}
x \le 200
\end{equation}
in order to obtain the exact modeling of $P_1 \cup P_2$. \quad $\triangleleft$
\end{example}

We also see an analogy between the relationship between our work and that of Williams, Hooker,  and Balas and the two main approaches to algorithmically solving integer programs -- branching and cutting planes. Disjunction is the organizing concept of branch-and-bound methods in integer programming, which is also at the core of the work of Williams, Hooker and Balas. By contrast, cutting planes in integer programming often result from ``rounding'', which introduces ceiling (or floor) operations. This is in concert with our approach to describing mixed-integer sets with \ch functions. Indeed, our main result relies on results that also serve as a foundation for the theory of cutting planes. Our requirements, however, are more demanding than standard integer programming since we solve parametrically in the right-hand side. Hence, we add all possible cutting planes of interest for any right-hand side (see Theorem~\ref{theorem:b-j-mod} and cf. Theorem~23.4 of \cite{schrijver86}). This full complement of ``redundant constraints'' are needed for the projection to work, as discussed in Section~\ref{s:consistency-testers}.


\section*{Acknowledgments}

We thank the anonymous reviewers who commented on a preliminary version of this manuscript that appears in the 2017 proceedings of the IPCO conference (referenced here as \cite{basu2017mixed}). The first author is supported by NSF grant CMMI1452820. The third author thanks the University of Chicago for its generous research support.

\bibliographystyle{plain}
\bibliography{../../../../references/references}

\end{document}